\documentclass[final,reqno]{elsarticle}
\usepackage{natbib}
\usepackage{mathrsfs}
\usepackage{amssymb,amsmath,amsfonts,latexsym}
\usepackage{amsthm}
\usepackage{graphicx,float,epsfig,color,fancyhdr}
\usepackage{colortbl}
\usepackage[colorlinks=false, linkcolor=blue,  citecolor=OliveGreen,  pdfstartview={}{}]{hyperref} 
\usepackage[dvipsnames]{xcolor}

\newtheorem{lemma}{Lemma}[section]
\newtheorem{remark}{Remark}[section]

\newtheorem{theorem}{Theorem}[section]

\newtheorem{problem}{Problem}

\def\CE{\mathcal{E}}
\def\CM{\mathcal{X}}
\def\CN{\mathcal{Y}}

\def\CT{\mathcal{T}}
\def\E{K}
\def\G{\Gamma}
\def\Go{\G_0}
\def\HusE{H^{1+s}(\E)}
\def\HusO{H^{1+s}(\O)}
\def\HuE{H^1(\E)}
\def\HuGoO{H^1_{\Go}(\O)}
\def\HuO{H^1(\O)}
\def\HurO{H^{1+r}(\O)}

\def\HusO{H^{1+s}(\O)}

\def\LO{L^2(\O)}
\def\N{\mathbb{N}}
\def\O{\Omega}
\def\P{\mathbb{P}}
\def\R{\mathbb{R}}
\def\dim{\mathop{\mathrm{\,dim}}\nolimits}

\def\g{\gamma}
\def\ha{\widehat{\boldsymbol{a}}}

\def\hdel{\widehat{\delta}}

\def\l{\lambda}
\def\PiK{\Pi^{\E}}

\def\sp{\mathop{\mathrm{sp}}\nolimits}

\def\Vh{V_h}
\def\VK{V^{\E}}
\newcommand{\vertiii}[1]{{\left\vert\kern-0.25ex\left\vert\kern-0.25ex\left\vert #1 
    \right\vert\kern-0.25ex\right\vert\kern-0.25ex\right\vert}}
\journal{}
\date{\today}

\begin{document}
\begin{frontmatter}

\title{A virtual element method for the Steklov
eigenvalue problem allowing small edges.}

\author[1]{Felipe Lepe}
\ead{flepe@ubiobio.cl}
\address[1]{Departamento de Matem\'atica,
Universidad del B\'io-B\'io, Casilla 5-C, Concepci\'on, Chile.}
\author[1,2]{David Mora}
\ead{dmora@ubiobio.cl}
\address[2]{CI$^2$MA, Universidad de Concepci\'on, Casilla 160-C, Concepci\'on, Chile.}
\author[3]{Gonzalo Rivera}
\ead{gonzalo.rivera@ulagos.cl}
\address[3]{Departamento de Ciencias Exactas, Universidad de Los Lagos,
Casilla 933, Osorno, Chile.}
\author[2]{Iv\'an Vel\'asquez}
\ead{ivelasquez@ci2ma.udec.cl}

\begin{abstract} 
The aim of this paper is to analyze the influence 
of small edges in the computation 
of the spectrum of the Steklov eigenvalue problem by a lowest order
virtual element method.
Under weaker assumptions on the polygonal meshes,
which can permit arbitrarily small edges with respect
to the element diameter,
we show that the scheme provides a correct approximation
of the spectrum and prove optimal error estimates
for the eigenfunctions and a double order for the eigenvalues.
Finally, we report some numerical tests
supporting the theoretical results.


\end{abstract}


\end{frontmatter}


\setcounter{equation}{0}
\section{Introduction}

In this paper we are interested in the approximation
by virtual elements of the eigenvalues and eigenfunctions
of the Steklov problem which is characterized by
the presence of the eigenvalue on the boundary condition. 
This problem has attracted much attention
in recent years due to the important applications
in many physical subjects.
For instance, it appears in the study of the dynamic
of liquids in moving containers, the so called
sloshing problem \cite{BRS03, CM79, CY96}.
Also, this problem have interesting applications
in inverse scattering \cite{LSTJSC19},
among other.


There are several studies on the finite element approximations
of the Steklov eigenvalue problem, for example, see
\cite{A_M2AN2004, AP_APNUM2009, BrOs1972, GM-ima2011, LLX_apmath2013, YLL_apnum2009, XIMA13}.
Traditionally, finite element methods rely on triangular (simplicial)
and quadrilateral meshes. However, in complex simulations
one often encounters general polygonal and polyhedral meshes.
In recent years there has been a significant growth
in the mathematical and engineering literature in developing
numerical methods that can make use of general polytopal
meshes; among the large number of papers on this subject, we cite as a minimal sample
\cite{BBCMMR2013,BLMbook2014,CGH14,DDbook2020,RW,ST04}.



The VEM has been introduced in \cite{BBCMMR2013}
and has been applied successfully in a large
range of problems arising from engineering and
physics phenomenons; see for instance \cite{AN2020,ABSV2016,BDR2017,
BLM2015,BLV-M2AN,BBBPS2016,CG2017,CGPS2017,MPP2018,PPR15}.
Regarding VEM for eigenvalue problems, we mention the following recent works
\cite{CGMMV,DrAA2011,GMV, GV, MM, MR2019, MRR1, MRR2, MV2}.
In particular, an a priori and a posteriori VEM discretization
for the Steklov eigenvalue problem has been presented in \cite{MRR1,MRR2}.
However, the theoretical results and error estimates for the
eigenvalues and eigenfunctions were obtained
under the standard mesh assumptions introduced
in \cite{BBCMMR2013}, which do not allow
to consider meshes containing elements with small edges
compared to the element diameter.

In \cite{BLR17,BM20,BrS} has been recently
analyzed the possibility to consider in VEM discretizations
arbitrarily small edges with respect to the element diameter.
Here, we will follow the VEM approach presented in \cite{BLR17},
for the Poisson equation, to write a lowest order virtual scheme
for the Steklov eigenvalue problem which permit arbitrarily
small edges in the polygonal meshes.
This can be useful in adaptive schemes
by considering refined meshes as a tool to
handle solutions with corner singularities.

More precisely, we will propose a
virtual element method of lowest order to solve the Steklov
eigenproblem, allowing small edges in the polygons of the mesh. We will consider the continuous
variational formulation presented in \cite{MRR1};
however, we will write a different discrete
virtual scheme, which is based on a different
stabilization bilinear form (see \cite{WRR}).
We will use the so-called Babu\v ska-Osborn abstract
spectral approximation theory (see \cite{BO}), to show
that under weaker assumptions on the polygonal meshes,
the resulting virtual element scheme provides a correct
approximation of the spectrum and prove optimal order
error estimates for the eigenfunctions and a double
order for the eigenvalues.
In particular, our theoretical estimates fully support 
meshes with arbitrarily small edges with respect to
the element diameter. In addition, we remark
that spurious modes were not found for different values
of the a scaled stabilization (see in particular
Section~\ref{sectionslosh} below).
Moreover, the present work can be seen as a stepping
stone towards the more challenging eigenvalue problems.



The paper is organized as follows: In Section~\ref{SEC:STAT},
we present the model problem and preliminary 
results related to the solution operator and
eigenfunctions. More precisely, we will establish
the spectral characterization
of the solution operator, which allows to study
the numerical method. Section~\ref{SEC:DISCRETE}
is dedicated to present the 
virtual element method. Here we will introduce
the assumptions on the mesh.
We will present approximation results that will be
the key point of our analysis, which will depend on
the particular choice of the stabilization form.
Section~\ref{SEC:approximation}, contains the error estimates
for the eigenfunctions and a double order for the eigenvalues.
Finally, in Section~\ref{sec:numerics}
we present some numerical results on different families
of polygonal meshes with small edges, in order
to confirm the theoretical rates of convergence proved in the paper
and to confirm that it is not polluted with spurious modes.


Throughout the article we will use standard notations
for Sobolev spaces, norms and seminorms. Moreover, we
will denote by $C$ a generic constant independent of
the mesh parameter $h$, which may take different values in
different occurrences.

\label{SEC:INTR}


\setcounter{equation}{0}
\section{The spectral problem}
\label{SEC:STAT}

Let $\O\subset\R^2$ be a bounded domain with polygonal boundary
$\partial\O$. Let $\Go$ and $\G_1$ be disjoint open subsets of
$\partial\O$ such that $\partial\O=\bar{\G}_0\cup\bar{\G}_1$ and
$\left|\Go\right|\ne0$. We denote by $n$ the outward unit normal vector
to $\partial\O$ and by $\partial_n$ the normal derivative.

In what follows, we recall the variational formulation of the
Steklov eigenvalue problem proposed
in \cite{MRR1}. Also, we summarize some results from this reference.

The  Steklov eigenvalue problem reads as follows:
Find $(\l,u)\in\R\times\HuO$, $u\ne0$, such that
$$
\left\{\begin{array}{l}
\Delta u=0\quad\text{in }\O,
\\[0.1cm]
\partial_n u
=\left\{\begin{array}{ll}
\l u & \text{on }\Go,
\\
0 & \text{on }\G_1.
\end{array}\right.
\end{array}\right. 
$$
where $\partial_n u$ denotes the normal derivative of $u$.
By testing the first equation above with $v\in\HuO$ and integrating by
parts, we arrive at the following equivalent variational formulation:

\begin{problem}
\label{P1}
Find $(\l,u)\in\R\times\HuO$, $u\ne0$, such that
$$
\int_{\O}\nabla u\cdot\nabla v
=\l\int_{\Go}uv\qquad\forall v\in\HuO.
$$
\end{problem}

Observe that the left-hand side is not $\HuO$-elliptic. A remedy for this 
is  to use a shift argument to rewrite Problem \ref{P1}
in the following form: 
\begin{problem}
\label{P2}
Find $(\l,u)\in\R\times\HuO$, $u\ne0$, such that
$$
\ha(u,v)
=\left(\l+1\right)b(u,v)
\qquad\forall v\in\HuO,
$$
\end{problem}
where the bilinear form
$\ha:\HuO\times \HuO\rightarrow\mathbb{R}$
is defined by
\begin{align*}
\ha(u,v)
& :=a(u,v)+b(u,v)
\qquad u,v\in\HuO,
\end{align*}
with  $a,b:\HuO\times\HuO\rightarrow\mathbb{R}$ defined by 
\begin{equation*}
a(u,v)
 :=\int_{\O}\nabla u\cdot\nabla v,
\qquad b(u,v)
 :=\int_{\Go}uv
\qquad u,v\in\HuO.
\end{equation*}
All the previous bilinear forms are bounded and symmetric.
In addition, the next result, proved in \cite[Lemma 2.1]{MRR1},
establishes that $\ha(v,v)$ is $\HuO$-elliptic.
\begin{lemma}
\label{ha-elipt}
There exists a constant $\alpha>0$, depending on $\O$, such that
$$
\ha(v,v)
\ge\alpha\left\|v\right\|_{1,\O}^2
\qquad\forall v\in\HuO.
$$
\end{lemma}

Next, we define the solution operator associated with Problem~\ref{P2}:
\begin{align*}
T:\ \HuO & \longrightarrow\HuO,
\\
f & \longmapsto Tf:=w,
\end{align*}
where $w\in\HuO$ is the unique solution (as a consequence
of  Lemma~\ref{ha-elipt} and the Lax-Milgram Theorem)
of the following source problem:
\begin{equation}
\label{T1}
\ha(w,v)=b(f,v)
\qquad\forall v\in\HuO.
\end{equation}

%


Thus, the linear operator $T$ is well
defined and bounded. Also, $T$ is self-adjoint with respect to
the inner product $\ha(\cdot,\cdot)$ in $\HuO$ (see \cite[Section 2]{MRR1}). 

Notice that $(\l,u)\in\R\times\HuO$ solves
Problem~\ref{P2} (and hence Problem~\ref{P1}) if and only if $Tu=\mu u$
with $\mu\neq0$ and $u\ne0$, in which case $\mu:=1/(1+\l)$.

The following additional regularity result for the solution of
problem~\eqref{T1} and consequently, for the eigenfunctions of $T$, has been
proved in \cite[Lemma 2.2]{MRR1}.

\begin{lemma}
\label{LEM:REG} 
\begin{itemize}
\item[i)] for all $f\in\HuO$, there exist $r\in(1/2,1]$ and $C>0$ such that the
solution $w$ of problem~\eqref{T1} satisfies $w\in\HurO$
and
$$
\left\|w\right\|_{1+r,\O}
\le C\left\|f\right\|_{1,\O};
$$
\item[ii)] if $u$ is an eigenfunction of Problem~\ref{P1} with
eigenvalue $\l$, there exist $r>1/2$ and $C>0$ (depending on $\l$) such that
$u\in\HurO$ and
$$
\left\|u\right\|_{1+r,\O}
\le C\left\|u\right\|_{1,\O}.
$$
\end{itemize}
\end{lemma}
\begin{remark}\label{angleome}
The constant $r>1/2$ is the Sobolev exponent for the Laplace
problem with Neumann boundary conditions. If $\O$ is convex, then
$r\ge1$, whereas, otherwise, $r:=\pi/\omega$ with $\omega$
being the largest reentrant angle of $\O$ (see \cite{G}).
\end{remark}

Hence, as a consequence of the compact inclusion
$\HurO\hookrightarrow\HuO$, $T$ is a compact operator.
We have the following spectral
characterization for the operator $T$.

\begin{theorem}
\label{CHAR_SP}
The spectrum of $T$ decomposes as follows: 
$\sp(T)=\left\{0,1\right\}\cup\left\{\mu_k\right\}_{k\in\N}$, where:
\begin{enumerate}
\item[\textit{i)}] $\mu=1$ is an eigenvalue of $T$ and its associated
eigenspace is the space of constant functions in $\O$;
\item[\textit{ii)}] $\mu=0$ is an infinite-multiplicity eigenvalue of
$T$ with associated eigenspace is 
$\HuGoO:=\left\{q\in\HuO:\ q=0\ \mbox{on }\Go\right\}$;
\item[\textit{iii)}] $\left\{\mu_k\right\}_{k\in\N}\subset(0,1)$ is a
sequence of finite-multiplicity eigenvalues of $T$ which converge to $0$
and their corresponding eigenspaces lie in $\HurO$.
\end{enumerate}
\end{theorem}
%


\setcounter{equation}{0}
\section{VEM discretization}
\label{SEC:DISCRETE}

We will study in this section, the virtual element
numerical approximation of the eigenproblem
presented in Problem~\ref{P2}, by considering weaker mesh assumptions
than the mesh assumptions considered in \cite{MRR1}.
%
We will follow some
recent results from \cite{BLR17,BrS} for the Poisson problem.
With this aim, first we recall the mesh construction.


Let $\left\{\CT_h\right\}_h$ be a sequence of decompositions of $\O$
into polygons $\E$. Let $h_\E$ denote the diameter of the element $\E$
and $h$ the maximum of the diameters of all the elements of the mesh,
i.e., $h:=\max_{\E\in\O}h_\E$. For $\CT_h$ we will consider
the following assumption:
\begin{itemize}
\item \textbf{A1.} There exists $\g>0$ such that, for all meshes
$\CT_h$, each polygon $\E\in\CT_h$ is star-shaped with respect to a ball
of radius greater than or equal to $\g h_{\E}$.
\end{itemize}

The next results will be obtained only under assumption \textbf{A1.}
In particular, we can consider meshes with edges arbitrarily small
with respect to the element diameter $h_{\E}$.

%
We consider now a simple polygon $\E$, we define 
$$
\mathbb{B}_1(\partial\E)
:=\left\{v\in C^0(\partial\E):v|_e\in\P_1(e)
\text{ for all edges }e\subset\partial\E\right\}.
$$
We then consider the finite-dimensional space defined as follows:
$$
\VK
:=\left\{v\in\HuE:
\ v|_{\partial\E}\in\mathbb{B}_1(\partial\E)
\text{ and }\Delta v|_{\E}=0\right\}.
$$

As in \cite{BLR17}, we choose the following degrees of freedom:
For all $v_h\in\VK$, they are defined as follows:
\begin{itemize}
\item values of $v_h$ at the $N_K$ vertices of $\E$.
\end{itemize}

Next, for every
decomposition $\CT_h$ of $\O$ into simple polygons $\E$, we define
the global virtual space
$$
\Vh:=\left\{v\in\HuO:\ v|_{\E}\in\VK\right\}.
$$

In order to construct the discrete scheme, we need some preliminary
definitions. First, we split the bilinear form $\ha(\cdot,\cdot)$ as
follows:
$$
\ha(w,v)=\sum_{\E\in\CT_h}a^{\E}(w,v)+b(w,v)
\qquad w,v\in\HuO,
$$
where 
\begin{equation*}
\label{defin}
a^{\E}(w,v)
:=\int_{\E}\nabla w\cdot\nabla v
\qquad w,v\in\HuO.
 \end{equation*}
 


Next, for any $\E\in\CT_h$ and for any sufficiently regular
function $v$, we define first
\begin{equation}
\label{eq:fixed}
\overline{v}:=\vert\partial\E\vert^{-1}\int_{\partial \E}v.
\end{equation}

Now, we define the projector $\PiK:\VK\longrightarrow\P_1(\E)\subseteq\VK$ for
each $v\in\VK$ as the solution of 
\begin{subequations}
\begin{align*}
a^{\E}\big(\PiK v,q\big)
& =a^{\E}(v,q)
\qquad\forall q\in\P_1(\E),
\\
\overline{\PiK v}
&=\overline{v}.
\end{align*}
\end{subequations}


Now, we introduce the following symmetric and semi-positive
definite bilinear form on $\VK\times\VK$ (see \cite{WRR}).
For all elements $\E\in\CT_h$:
\begin{equation}
\label{20}
S^{\E}(w_h,v_h):=h_{\E}\int_{\partial{\E}}\partial_{s}w_h\partial_s
v_h\qquad\forall w_h,v_h\in\VK, 
\end{equation}
where $\partial_s$ denotes a derivative along the edge.

Then, set
$$
a_h(w_h,v_h)
:=\sum_{\E\in\CT_h}a_h^{\E}(w_h,v_h)
\qquad w_h,v_h\in\Vh,
$$
where $a_h^{\E}(\cdot,\cdot)$ is the bilinear form defined on
$\VK\times\VK$ by
\begin{equation}
\label{21}
a_h^{\E}(w,v)
:=a^{\E}\big(\PiK w,\PiK v\big)
+S^{\E}\big(w-\PiK w,v-\PiK v\big)
\qquad w,v\in\VK. 
\end{equation}

%

Now, we introduce the following discrete semi-norm:
\begin{equation}
\label{eq:triple}
\vertiii{v}_{\E}^2:=a^{\E}\big(\PiK v,\PiK v)+S^{\E}(v-\bar{v},v-\bar{v})\qquad\forall
v\in\VK+\mathcal{V}^{\E},
\end{equation}
where  $\mathcal{V}^{\E}\subseteq H^1(\E)$ is a subspace of sufficiently
regular functions for $S^{\E}(\cdot,\cdot)$ to make sense.
 
Now, for any sufficiently regular functions,
we introduce the following global semi-norms
$$
\vertiii{v}^2:=\sum_{\E\in\CT_h}\vertiii{v}_{\E}^2,\qquad
\left|v\right|_{1,h}^2
:=\sum_{\E\in\CT_h}\left\|\nabla v\right\|_{0,\E}^2.
$$
 
It has been proved in  \cite[Lemma~3.1]{BLR17} that
there exist positive constants $C_1,C_2,C3$, independent of $h$,
such that
\begin{align}
C_1\vertiii{v}_{\E}^2\le a_h^{\E}(v,v)\le C_2\vertiii{v}_{\E}^2\quad\forall v\in\VK,\label{eqrefgt1}\\
a_h^{\E}(v,v)\le C_3(\vertiii{v}^2+\vert v\vert_{1,\E}^2)\quad\forall v\in\VK.\label{eqrefgt12}
\end{align}
In addition, it holds
\begin{align}
a^{\E}(v,v)\le C_4\vertiii{v}_{\E}^2\quad\forall v\in\VK,\label{eqrefgt2}\\
\nonumber\vertiii{p}_{\E}^2\le C_5 a^{\E}(p,p)\quad\forall p\in\P_1(\E),\label{eqrefgt3}
\end{align}
where $C_4,C_5$ are independent of $h$.

Now we are in a position to write the virtual element
discretization of Problem~\ref{P1}.
\begin{problem}
\label{P11}
Find $(\l_h,u_h)\in\R\times\Vh$, $u_h\ne0$, such that
$$
a_h(u_h,v_h)
=\l_h b(u_h,v_h)
\qquad\forall v_h\in\Vh.
$$
\end{problem}

We use again a shift argument to rewrite this discrete eigenvalue
problem in the following convenient equivalent form.
\begin{problem}
\label{P21}
Find $(\l_h,u_h)\in\R\times\Vh$, $u_h\ne0$, such that
$$
\ha_h(u_h,v_h)
=\left(\l_h+1\right)b(u_h,v_h)
\qquad\forall v_h\in\Vh,
$$
\end{problem}

\noindent where the bilinear form
$\ha_h:\Vh\times\Vh\rightarrow\mathbb{R}$ is defined by
\begin{equation*}
\label{eq:shifteado}
\ha_h(u_h,v_h)
:=a_h(u_h,v_h)+b(u_h,v_h)
\qquad u_h,v_h\in\Vh.
\end{equation*}

Clearly $\ha_h(\cdot,\cdot)$ is symmetric and continuous.
In the following result we prove that $\ha_h(\cdot,\cdot)$ is elliptic
in $\Vh$.

\begin{lemma}
\label{ha-elipt-disc}
There exists a constant $\beta>0$, independent of $h$, such that
$$
\ha_{h}(v_h,v_h)
\ge\beta\left\|v_h\right\|_{1,\O}^2
\qquad\forall v_h\in\Vh.
$$
\end{lemma}

\begin{proof}
From  the definition of the bilinear form $\ha_h(\cdot,\cdot)$, we have that
\begin{multline*}
\ha_{h}(v_h,v_h)=a_h(v_h,v_h)+b(v_h,v_h)\\
=\sum_{\E\in\CT_h}a_h^{\E}\big(v_h,v_h\big)+b(v_h,v_h)
\geq\sum_{\E\in\CT_h}C_1\vertiii{v}_{\E}^2+\Vert v_h\Vert_{0,\Go}^{2}
\geq C\sum_{\E\in\CT_h}a^{\E}\big(v_h,v_h\big)+\Vert v_h\Vert_{0,\Go}^{2}\\
\geq C\vert v_h\vert_{1,\O}^2+\Vert v_h\Vert_{0,\Go}^{2}
\geq \beta\left\|v_h\right\|_{1,\O}^2
\qquad\forall v_h\in\Vh,
\end{multline*}
where we have used \eqref{eqrefgt1}, \eqref{eqrefgt2} and
the generalized Poincar\'e inequality. This concludes the proof.
%
\end{proof}

With this coercivity result at hand, we are in a position
to  introduce the  discrete solution operator
\begin{align*}
T_h:\ \HuO & \longrightarrow\HuO,
\\
f & \longmapsto T_hf:=w_h,
\end{align*}
where $u_h\in\Vh$ is the solution of the following discrete source
problem
\begin{equation*}
\ha_h(w_h,v_h)=b(f,v_h)
\qquad\forall v_h\in\Vh.
\end{equation*}

 Notice that Lemma~\ref{ha-elipt-disc} implies that the linear operator $T_h$ is well
defined and bounded uniformly with respect to $h$. Moreover, as in the
continuous case, $(\l_h,u_h)\in\R\times\Vh$ solves Problem~\ref{P21}
(and hence Problem~\ref{P11}) if and only if $T_hu_h=\mu_h u_h$ with
$\mu_h\neq0$ and $u_h\ne0$, in which case $\mu_h:= 1/(1+\l_h)$.
Also, $T_h|_{\Vh}:\ \Vh\longrightarrow\Vh$ is self-adjoint with
respect to $\ha_{h}(\cdot,\cdot)$. 

As a consequence, we have the following spectral characterization for $T_h$.

\begin{theorem}
\label{CHAR_SP_DISC}
The spectrum of $T_h|_{\Vh}$ consists of $M_h:=\dim(\Vh)$ eigenvalues,
repeated according to their respective multiplicities. It decomposes as
follows: $\sp(T_h|_{\Vh})
=\left\{0,1\right\}\cup\left\{\mu_{hk}\right\}_{k=1}^{N_h}$, where:
\begin{enumerate}
\item[\textit{i)}] the eigenspace associated with $\mu_h=1$ is the space
of constant functions in $\O$;
\item[\textit{ii)}] the eigenspace associated with $\mu_h=0$ is
$Z_h:=\Vh\cap\HuGoO=\left\{q_h\in V_h:\ q_h=0\ \mbox{on }\Go\right\}$;
\item[\textit{iii)}] $\mu_{hk}\subset(0,1)$,
$k=1,\dots,N_h:=M_h-\dim(Z_h)-1$, are non-defective eigenvalues repeated
according to their respective multiplicities.
\end{enumerate}
\end{theorem}


\setcounter{equation}{0}
\section{Convergence and error estimates}
\label{SEC:approximation}

In order to prove that the solutions of the discrete
problem converge to those of the continuous
problem, we will follow the standard procedure
for spectral theory for compact operators \cite{BO},
which consist in showing that $T_h$
converges in norm to $T$ as $h$ tends to zero.



With this end, we begin by proving the following result.


\begin{lemma}
\label{lemcotste}
There exists $C>0$, independent of $h$, such that,
for all $f\in\HuO$, if $w=Tf$ and $w_h=T_hf$, then
\begin{equation}
\label{eq:T-Th1}
\left\|w-w_h\right\|_{1,\O}
\le C\left(\left\|w-w_I\right\|_{1,\O}+ \left|w_{\pi}-w\right|_{1,h}
+\vertiii{w-w_{I}}+\vertiii{w-w_{\pi}}\right),
\end{equation}
for all $w_I\in\Vh$ and for all $w_{\pi}\in\LO$ such that
$w_{\pi}|_{\E}\in\P_1(\E)$ $\forall\E\in\CT_h$.  In addition 
\end{lemma}
\begin{equation}
\label{eq:uh-uI}
\vertiii{w_{h}-w_{I}}
\le C\left(\left\|w-w_I\right\|_{1,\O}
+ \left|w_{\pi}-w\right|_{1,h}+\vertiii{w-w_{I}}+\vertiii{w-w_{\pi}}\right).
\end{equation}

\begin{proof}
Let $w=Tf$ and
$w_h=T_hf$. From triangular inequality we have
$$\|w-w_{h}\|_{1,\O}\leq \|w-w_{I}\|_{1,\O}+\|w_{I}-w_{h}\|_{1,\O}. $$

Our task is to estimate the norms of the right hand side above.
To do this, we will consider the arguments
on the proof of Lemma~\ref{ha-elipt-disc}.  

Now, for $w_I\in\Vh$, we set $v_h:=w_h-w_I$ and thanks to
Lemma~\ref{ha-elipt-disc}, the definitions of $a_h^{\E}(\cdot,\cdot)$ (cf. \eqref{21})
and those of $T$ and $T_h$, we have 
\begin{align*}
\left(\vertiii{v_{h}}+\left\|v_h\right\|_{0,\G_{0}}\right)^{2}&\leq
2\left(\vertiii{v_{h}}^{2}+\left\|v_h\right\|^2_{0,\G_{0}}\right)\\
&\leq C\ha_h(v_h,v_h)=C(\ha_h(w_h,v_h)-\ha_h(w_I,v_h)) 
\\
&=C\left(b(f,v_h)-\sum_{\E\in\CT_h}a_h^\E(w_I,v_h)-b(w_I,v_h)\right)
\\
&=C\Big(b(f,v_h)-b(w_I,v_h)\\
&\left.-\sum_{\E\in\CT_h}
\left(a_h^\E(w_I-w_{\pi},v_h)+a^\E(w_{\pi}-w,v_h)+a^\E(w,v_h)\right)\right)
\\
& =C\left(b(w-w_I,v_h)
-\sum_{\E\in\CT_h}
\left(a_h^\E(w_I-w_{\pi},v_h)+a^\E(w_{\pi}-w,v_h)\right)\right).
\end{align*}
Therefore, from the trace theorem, \eqref{eqrefgt1} and the boundedness of
$a_h^\E(\cdot,\cdot)$ (\eqref{eqrefgt12}) and $a^\E(\cdot,\cdot)$, we get
\begin{align*}
\left(\vertiii{v_{h}}+\left\|v_h\right\|_{0,\Go}\right)^{2}&
\leq C\Big(\left\|w-w_I\right\|_{0,\Go}
\left\|v_h\right\|_{0,\Go}\\
&\quad\left.+\sum_{\E\in\CT_h}
\left(C(\vertiii{w_I-w_{\pi}}_{\E}+\left|w_I-w_{\pi}\right|_{1,\E})
\vertiii{v_h}_{\E}
+\left|w_{\pi}-w\right|_{1,\E}
\left|v_h\right|_{1,\E}\right)\right)
\\
&  \leq C\left(\left\|w-w_I\right\|_{0,\Go}
\left\|v_h\right\|_{0,\Go}
+\sum_{\E\in\CT_h}
\left\{C\left(\vertiii{w_I-w_{\pi}}_{\E}+\left|w_{\pi}-w\right|_{1,\E}\right)
\vertiii{v_h}_{\E}\right\}\right)
\\
&  \leq C \left(\left\|w-w_I\right\|_{0,\Go}+\vertiii{w_I-w_{\pi}}+ \left|w_{\pi}-w\right|_{1,h}\right)\left(\|v_{h}\|_{0,\Go}+\vertiii{v_{h}}\right).
\end{align*}
Therefore, we have 
\begin{equation*}
\vertiii{v_{h}}+\left\|v_h\right\|_{0,\G_{0}}
\leq \widetilde{C} \left(\left\|w-w_I\right\|_{1,\O}
+\vertiii{w-w_{\pi}}+\vertiii{w-w_{I}}+ \left|w_{\pi}-w\right|_{1,h}\right).
\end{equation*}
Finally, \eqref{eq:T-Th1} follows from the
triangular inequality and the generalized Poincar\'e inequality.
Moreover, \eqref{eq:uh-uI} follows from the above estimate.
\end{proof}


Let us introduce the following approximation result for
polynomials in star-shaped domains (see for instance \cite{BS-2008}),
which is derived by results of interpolation between Sobolev spaces
(see for instance \cite[Theorem~I.1.4]{GR}), leading to an analogous
result for integer values of $s$. Moreover, we remark that the
result for integer values is stated in
\cite[Proposition~4.2]{BBCMMR2013} and follows from the well establish
Scott-Dupont theory (see \cite{BS-2008}). 

\begin{lemma}
\label{lmm:bh}
If assumption {\bf A1} is satisfied, then there exists a constant
$C$, depending only on $k$ and $\g$, such that for every $s$ with 
$0\le s\le k$ and for every $v\in\HusE$, there exists
$v_{\pi}\in\P_k(\E)$ such that
$$
\left\|v-v_{\pi}\right\|_{0,\E}
+h_{\E}\left|v-v_{\pi}\right|_{1,\E}
\le Ch_{\E}^{1+s}\left\|v\right\|_{1+s,\E}.
$$
\end{lemma}

Now, we have the following approximation result
in the virtual space $V_h$, which follows from
\cite[Theorem~3.4]{BLR17}.


\begin{lemma}
\label{estima2}
Under the assumption  {\bf A1}, then, for
each $s$ with $1/2<s\le 1$, there exist $\widehat{t}>1/2$
and a constant $C$, independent of $h$,
such that for every $v\in\HusO$, there exists
$v_I\in\Vh$ that satisfies
\begin{align}
\label{firststimate}
\left|v-v_I\right|_{1+t,\E}&\le Ch_{\E}^{s-t}\left|v\right|_{1+s,\E},
\qquad 0\leq t<\min\{\widehat{t},s\},\\ 
\label{secondstimate}
\left\|v-v_I\right\|_{0,\E}
&\le Ch_{\E}\left|v\right|_{1+s,\E}.
\end{align}
\end{lemma}
\begin{proof}
Estimate \eqref{firststimate} has been obtained in \cite[Theorem~3.4]{BLR17}.
 To obtain \eqref{secondstimate}, with $1/2<s\leq1$,
first we use the Poincar\'e and the Cauchy-Schwarz inequalities,
to obtain (see \cite[Remark 4.1]{BM20}),

\begin{align*}
\left\|v-v_I\right\|_{0,\E}&\leq C\left(\int_{\partial \E}|v-v_{I}|ds+h_{\E}|v-v_{I}|_{1,\E}\right)\\
&\leq C\left(h_{\E}^{1/2}\|v-v_{I}\|_{0,\partial\E}+h_{\E}|v-v_{I}|_{1,\E}\right)\\
&\leq C\left(h_{\E}^{1/2}h_{\E}^{1/2}|v|_{1/2,\partial \E}+h_{\E}|v-v_{I}|_{1,\E}\right)\\
&\leq C\left(h_{\E}|v|_{1,\E}+h_{\E}|v-v_{I}|_{1,\E}\right)\\
&\leq Ch_{\E}|v|_{1+s,\E},
\end{align*}
where we have use an standard approximation estimate in one
dimension,
since $v_I|_{\partial\E}$ corresponds to the standard
piecewise linear Lagrange interpolant of $v$
and then $|v|_{1/2,\partial \E}\le|v|_{1,\E}$. This concludes the proof.
\end{proof}

Now we are in position to establish the convergence in norm of $T_h$ to $T$ as
$h\to0$.
\begin{lemma}
\label{convnorm}
There exists $r\in(1/2,1]$ (cf. Lemma~\ref{LEM:REG}(i)) and $C>0$, 
independent of $h$, such that
$$
\left\|\left(T-T_h\right)f\right\|_{1,\O}
\le Ch^{r}\left\|f\right\|_{1,\O}
\qquad\forall f\in\HuO.
$$
\end{lemma}

\begin{proof}
The result follows from Lemma~\ref{lemcotste}. In particular, we
have to bound the term on the right and side of \eqref{eq:T-Th1}.
For the first and second terms, using Lemmas \ref{estima2} and
\ref{lmm:bh}, respectively, we obtain
\begin{equation}
\label{eq:uI}
||w-w_{I}||_{1,\O}+|w-w_{\pi}|_{1,h}\leq
C\sum_{\E\in\CT_{h}}h_{\E}^{r}|w|_{1+r}\leq Ch^{r}|w|_{1+r,\O}.
\end{equation}
Now, we bound the term $\vertiii{w-w_{I}}$. To do this task, we 
invoke the definition of $\vertiii{\cdot}$ given in \eqref{eq:triple}, \eqref{eq:fixed}
and, operating as in the proof of \cite[Theorem 4.5]{BLR17} we have
\begin{align}
\nonumber
\sum_{\E\in\CT_{h}}\vertiii{w-w_I}_{\E}^{2}&
=\sum_{\E\in\CT_{h}}\left\{a^K\left(\PiK(w-w_I), \PiK(w-w_I) \right)\right.\\\nonumber
&\left.+S^K\left((w-w_I)-\overline{(w-w_{I})},(w-w_I)-\overline{(w-w_{I})}\right)\right\}\\\nonumber
&\leq C\sum_{\E\in\CT_{h}}\left(|\PiK(w-w_I)|^2_{1,K}+S^\E(w-w_I,w-w_I) \right)\\\label{eq:inter}
&=C\sum_{\E\in\CT_{h}}\left(|w-w_I|^2_{1,K}+h_K|w-w_I|_{1,\partial \E}^2\right).
\end{align}
Let $\sigma$ be such that $1/2<\sigma<r$, using a scaled trace inequality, we get
$$h_K|w-w_I|_{1,\partial \E}^2\leq C \left(|w-w_{I}|_{1,\E}^{2}
+h^{2\sigma}|w-w_{I}|_{1+\sigma,\E}^{2}\right)\leq Ch_{\E}^{2r}|w|_{1+r,\E}^{2}.$$  
Using the above estimate in \eqref{eq:inter} and
Proposition \ref{estima2}, we obtain
\begin{align}
\label{eq:estimeui}
\vertiii{w-w_{I}}=\left(\sum_{\E\in\CT_{h}}\vertiii{w-w_I}_{\E}^{2}\right)^{1/2}
&\leq C\left(\sum_{\E\in\CT_{h}}h_{\E}^{2r}|w|_{1+r,\E}^{2}\right)\leq Ch^{r}|w|_{1+r,\O}.
\end{align}
Similarly, we obtain
\begin{align*}
\vertiii{w-w_\pi}\leq Ch^{r}|w|_{1+r,\O}.
\end{align*}
thus, the lemma follows from \eqref{eq:uI}--\eqref{eq:estimeui}
and Lemma~\ref{LEM:REG}(i).
\end{proof}



We conclude the analysis of our paper deriving error estimates for our method.
In particular, we are going to present error estimates for eigenfunctions
and eigenvalues. With this aim, with Lemma~\ref{convnorm} at hand,
we will prove that isolated parts of $\sp(T)$ are approximated
by isolated parts of $\sp(T_h)$ (see \cite{K}).

Let $\mu\in(0,1)$ be an isolated eigenvalue of $T$ with
multiplicity $m$ and let $\CE$ be its associated eigenspace. Then, there
exist $m$ eigenvalues $\mu^{(1)}_h,\dots,\mu^{(m)}_h$ of $T_h$ (repeated
according to their respective multiplicities) which converge to $\mu$. From now and on, let $\CE_h$ be the discrete subspace associated to $\CE$, corresponding to the direct sum of their corresponding associated eigensapaces.


We recall the definition of the \textit{gap} $\hdel$ between two closed
subspaces $\CM$ and $\CN$ of $\HuO$:
$$
\hdel(\CM,\CN)
:=\max\left\{\delta(\CM,\CN),\delta(\CN,\CM)\right\},
\quad\text{where}\quad
\delta(\CM,\CN)
:=\sup_{x\in\CM:\ \left\|x\right\|_{1,\O}=1}
\left(\inf_{y\in\CN}\left\|x-y\right\|_{1,\O}\right).
$$
The following error estimates for the approximation of eigenvalues and
eigenfunctions hold true.

\begin{theorem}
\label{gap}
There exists a strictly positive constant $C$ such that
\begin{align*}
\hdel(\CE,\CE_h) 
& \le C\g_h,
\\
\left|\mu-\mu_h^{(i)}\right|
& \le C\g_h,\qquad i=1,\dots,m,
\end{align*}
where
$$
\g_h
:=\sup_{f\in\CE:\ \left\|f\right\|_{1,\O}=1}
\left\|(T-T_h)f\right\|_{1,\O}.
$$
\end{theorem}

\begin{proof}
As a consequence of Lemma~\ref{convnorm}, $T_h$ converges in norm to $T$
as $h$ goes to zero. Then, the proof follows as a direct consequence of
Theorems~7.1 and 7.3 from \cite{BO}.
\end{proof}

The theorem above yields error estimates depending on $\g_h$. The next
step is to show an optimal order estimate for this term. 

\begin{theorem}
\label{order}
There exist $r\in(1/2,1]$ and a positive constant $C$
such that
\begin{equation*}
\label{cota}
\left\|(T-T_h)f\right\|_{1,\O}
\leq Ch^{r}
\left\|f\right\|_{1,\O}
\qquad\forall f\in\CE,
\end{equation*}
and, consequently,
\begin{equation*}
\label{cotgap}
\g_h\le Ch^{r}.
\end{equation*}
\end{theorem}
\begin{proof}
See \cite[Theorem 4.2]{MRR1}.
\end{proof}

The error estimate for the eigenvalue $\mu\in(0,1)$ of $T$ leads to an
analogous estimate for the approximation of the eigenvalue
$\l=\frac1{\mu}-1$ of Problem~\ref{P1} by means of the discrete
eigenvalues $\l_h^{(i)}:=\frac1{\mu_h^{(i)}}-1$, $1\le i\le m$, of
Problem~\ref{P11}. 

We are able to improve the convergence order of 
Theorem~\ref{gap} for the eigenvalues. The following result shows in fact
that the convergence order is double.

\begin{theorem}
\label{cotadoblepandeo}
There exist $r\in(1/2,1]$ and a positive constant $C$
such that
$$\
\left|\l-\l_h^{(i)}\right|
\le Ch^{2r}.
$$
\end{theorem}

\begin{proof}
Let $u_h$ be such that $(\l_h^{(i)},u_h)$ is a solution of
Problem~\ref{P11} with $\left\|u_h\right\|_{1,\O}=1$. According to
Theorems~\ref{gap} and\ref{order}, there exists a solution
$(\l,u)$ of Problem~\ref{P1} such that 
\begin{equation}
\label{fin4}
\left\|u-u_h\right\|_{1,\O}\le Ch^{r}.
\end{equation}

From the symmetry of the bilinear forms and the facts that 
$a(u,v)=\l b(u,v)$ for all $v\in\HuO$ (cf. Problem~\ref{P1}) and 
$a_h(u_h,v_h)=\l_h^{(i)}b(u_h,v_h)$ for all $v_h\in\Vh$ (cf.
Problem~\ref{P11}), we have
\begin{align*}
a(u-u_h,u-u_h)-\l b(u-u_h,u-u_h)
& =a(u_h,u_h)-\l b(u_h,u_h)
\\
& =\left[a(u_h,u_h)-a_h(u_h,u_h)\right]
-\left(\l-\l_h^{(i)}\right)b(u_h,u_h),
\end{align*}
from which we obtain the following identity:
\begin{equation}
\label{45}
\left(\l_h^{(i)}-\l\right)b(u_h,u_h)
=a(u-u_h,u-u_h)-\l b(u-u_h,u-u_h)
+\left[a_h(u_h,u_h)-a(u_h,u_h)\right].
\end{equation}

The next step is to bound each term on the right hand side above. The
first and the second ones are easily bounded from the continuity of
$a(\cdot,\cdot)$ and $b(\cdot,\cdot)$, the trace theorem and
\eqref{fin4} as follows
\begin{equation}
\label{terms_12}
\left|a(u-u_h,u-u_h)\right|
+\l\left|b(u-u_h,u-u_h)\right|
\le Ch^{2r}.
\end{equation}

For the last term on the right hand side of \eqref{45},
we consider $u_\pi\in\P_1(\E)$ and $u_I\in V_h$
such that Lemmas~\ref{lmm:bh} and \ref{estima2} hold true, respectively.
Using standard arguments, we have 
\begin{align*}
\left|a_h(u_h,u_h)-a(u_h,u_h)\right|&
=\left| \sum_{\E\in\CT_{h}} a_{h}^{\E}(u_h-u_\pi,u_h-u_\pi)-a^{\E}(u_h-u_\pi,u_h-u_\pi)\right|\\
& \leq\sum_{\E\in\CT_{h}} \vertiii{u_h-u_\pi}_{\E}^2+\sum_{\E\in\CT_{h}}|u_h-u_\pi|_{1,\E}^2\\
&=\vertiii{u_h-u_\pi}^2+|u_h-u_\pi|_{1,h}^2\\
&\leq C\left(\vertiii{u-u_h}^2+\vertiii{u-u_\pi}^2+|u-u_h|^2_{1,h}+|u-u_\pi|^2_{1,h} \right)\\
&\leq C\left(\vertiii{u-u_I}^2+\vertiii{u_{h}-u_I}^2
+\vertiii{u-u_\pi}^2+|u-u_h|^2_{1,h}+|u-u_\pi|^2_{1,h} \right)\\
&\leq C\left(\vertiii{u-u_I}^2+\vertiii{u-u_\pi}^2
+\|u-u_{I}\|_{1,\O}^{2}+|u-u_{\pi}|_{1,h}^{2}+|u_h-u|^2_{1,\O} \right),
\end{align*}
where we have used \eqref{eq:uh-uI}.
Now, the terms on the right hand side above
can be bounded repeating
the argument in the proof of Lemma~\ref{convnorm}
and using the additional regularity result
in Lemma~\ref{LEM:REG}(ii). We get
\begin{align}
\label{ultimaeq}
 \left|a_h(u_h,u_h)-a(u_h,u_h)\right|\leq Ch^{2r}.
 \end{align}

On the other hand, by virtue of Lemma~\ref{ha-elipt-disc} and the fact
that $\l_h^{(i)}\to\l$ as $h$ goes to zero, we know that there exists
$C>0$ such that
$$
b(u_h,u_h)
=\frac{\ha_h(u_h,u_h)}{\l_h^{(i)}+1}
\ge\frac{\beta\left\|u_h\right\|_{1,\O}^2}{\l_h^{(i)}+1}
\ge\frac{\beta}{C}>0.
$$
Finally, the proof follows from \eqref{45}, by
using the above estimate together with
\eqref{terms_12} and \eqref{ultimaeq}.

\end{proof}
\section{Numerical experiments}

\label{sec:numerics}
In the present section we will report some numerical
tests in order to asses the performance of the proposed lowest order VEM
with meshes allowing small edges. All the reported numerical
results have been obtained with a MATLAB code.
In order to observe the performance and accuracy of the proposed method,
we will consider different computational domains,
where the eigenfunctions, on one hand, can be smooth enough and,
on the other, can be singular due to the non-convex domains.
 
For all the tests, we will report the computed eigenvalues
for different polygonal meshes and the order of convergence.
Our results will be compared with some references
and exact solutions in the cases where it is available.
In the cases where it is not possible to have a close form of the solution,
we will present extrapolated values for the eigenvalues (see \eqref{eq:fitting}).

\subsection{Square domain: the sloshing problem.}\label{sectionslosh}

We begin with a convex domain. In this case we consider
$\Omega:=(0,1)^2$ as computational domain. 
We fix $\Gamma_0$ on the top of the boundary
(representing a free surface)
and $\Gamma_1$ will the rest of the boundary.
In Figure~\ref{FIG:SLOSH}, we present the physical configuration
of the problem.
\vspace*{0.25cm}
\begin{figure}[H]
\begin{center}
\input{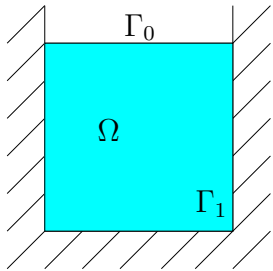}
\caption{Sloshing in a square domain.}
\label{FIG:SLOSH}
\end{center}
\end{figure}


For this problem there are analytical solutions of the form
\begin{equation}
\label{eq:exact_sloshing}
\lambda_n=n\pi\tanh(n\pi), \quad u_n(x,y)=\cos(n\pi x)\sinh (n\pi x), \quad n\in\mathbb{N}.
\end{equation}

In Figure~\ref{FIG:meshes}, we present plots of the polygonal meshes that
we will consider for our tests. We note that the family of polygonal meshes $\CT_h^1$
have been obtained by gluing two different polygonal meshes at $y=0.6$.
It can be seen that very small edges compared with the element
diameter appears on the interface of the resulting mesh.
The second family of polygonal meshes $\CT_h^2$ have been obtained
from a triangular mesh with an additional point on
each edge as a new degree of freedom which has been moved
to a distance $h_e^{2}$ from one vertex and $(h_{e}-h_{e}^{2})$
from the other. We observe that this family satisfy $\mathbf{A1}$
but fail to satisfy the usual assumption that distance between
any two of its vertices is greater than or equal to $Ch_{\E}$ for each polygon,
since the length of the smallest edge is $h_e^2$, while the diameter
of the element is bounded above by a multiple of $h_e$.
The refinement level for the meshes will be denoted by $N$,
which corresponds to the number of subdivisions in the abscissae. 
\begin{figure}[H]
\begin{center}
\begin{minipage}{15cm}
\centering\includegraphics[height=6.5cm, width=6.5cm]{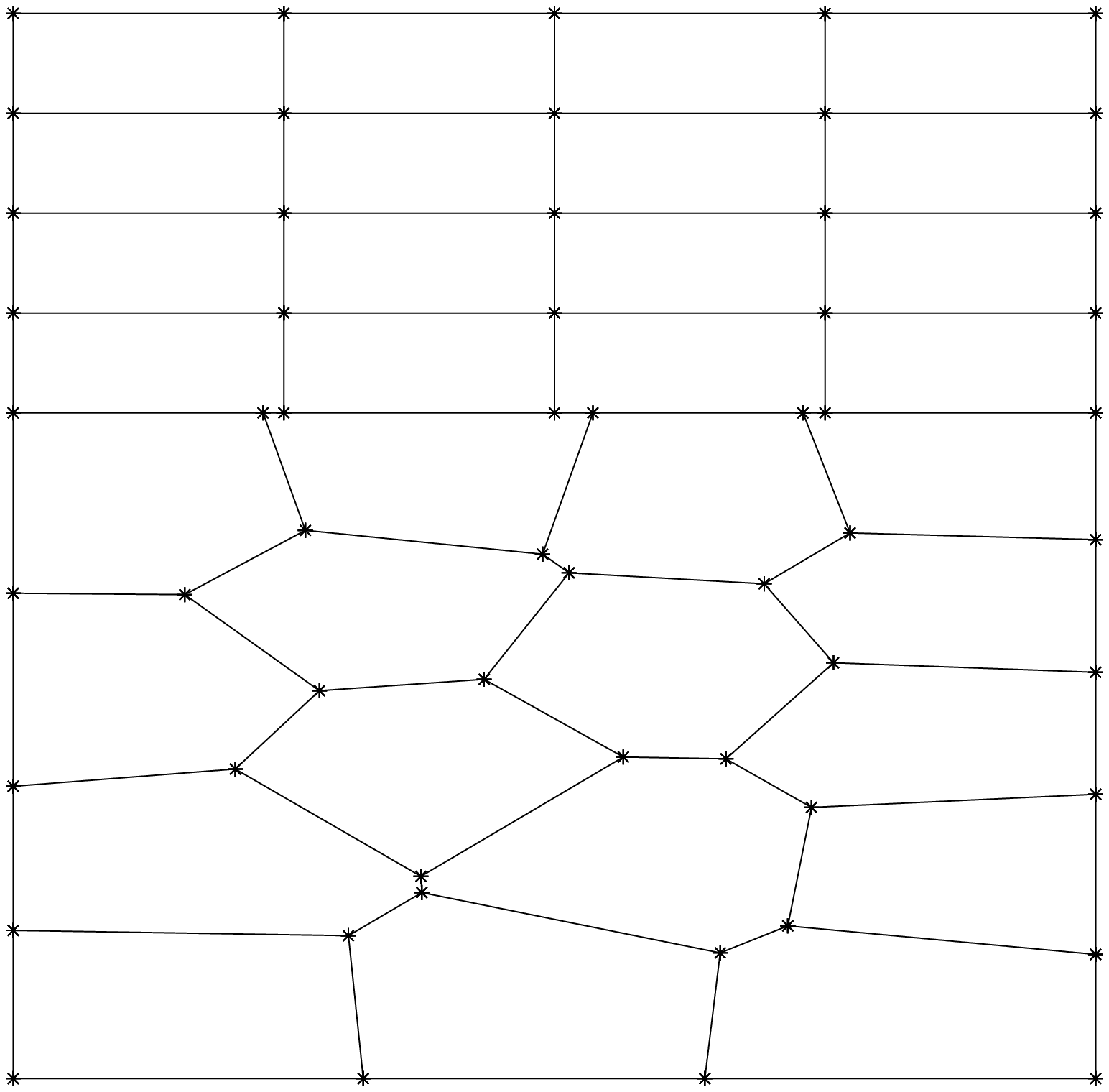}
\centering\includegraphics[height=6.5cm, width=6.5cm]{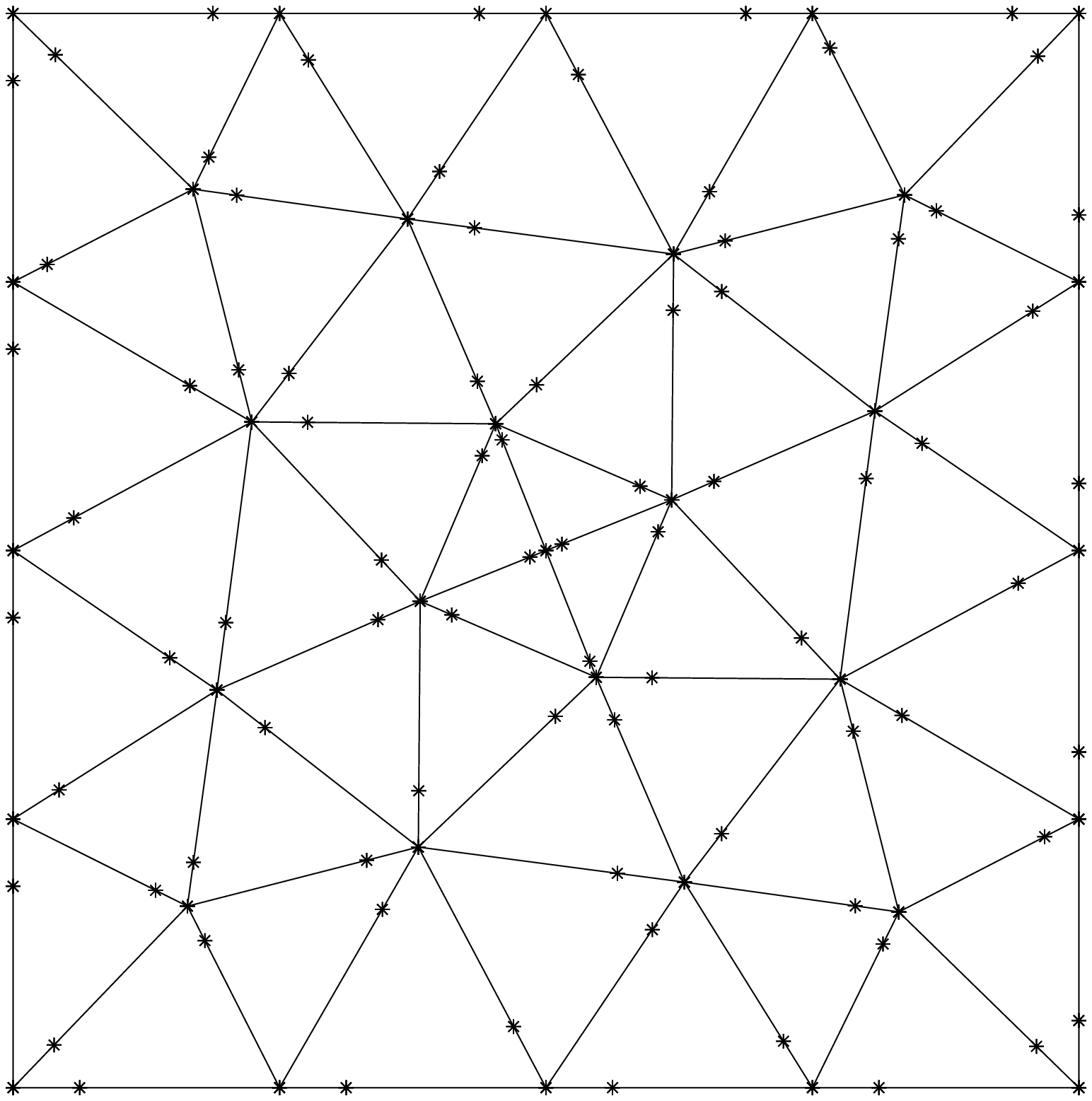}
\end{minipage}
\caption{Sample meshes with small edges. From left to right:
$\CT_h^1$  and $\CT_h^2$ for $N=4$.}
\label{FIG:meshes}
\end{center}
\end{figure}





In Table~\ref{tabla1}, we report the first six eigenvalues computed with
meshes $\CT_h^1$  and $\CT_h^2$. The row 'Order' reports the convergence
order of the eigenvalues, computed with respect to the exact ones obtained with
\eqref{eq:exact_sloshing}, which are presented in the row 'Exact'.

\begin{table}[H]
\begin{center}
\caption{Test1. The lowest computed eigenvalues
$\l_{hi}$, $1\le i\le 6$ for different meshes.}
\vspace{0.3cm}
\begin{tabular}{|c|c|c|c|c|c|c|}
\hline
\multicolumn{7}{|c|}{$\CT_{h}^{1}$} \\\hline
 $N$ & $\l_{h1}$ & $\l_{h2}$ & $\l_{h3}$ & 
 $\l_{h4}$ & $\l_{h5}$& $\l_{h6}$  \\
 \hline
8    &3.2422   &7.1802  &12.5152  &19.0595  &32.2326  &46.9310  \\
16   &3.1572   &6.4936  &10.1363  &14.2689  &19.0754  &24.7519  \\
32   &3.1366   &6.3347   &9.5984  &12.9787  &16.5158  &20.2508  \\
64   &3.1316   &6.2960   &9.4679  &12.6686  &15.9078  &19.1953  \\\hline
Order  &2.02   &2.04   &2.05   &2.00  & 2.12  & 2.12   \\\hline
Exact &3.1299  &  6.2831  &  9.4248  & 12.5664  & 15.7080 &  18.8496  \\
\hline
\multicolumn{7}{|c|}{$\CT_{h}^{2}$} \\\hline
 $N$ & $\l_{h1}$ & $\l_{h2}$ & $\l_{h3}$ & 
 $\l_{h4}$ & $\l_{h5}$& $\l_{h6}$  \\
 \hline
8    &3.1820  & 6.7247  &10.9616  &16.4556  &23.3067  &31.7699  \\
16   &3.1441  & 6.4043   &9.8511  &13.6063  &17.7749  &22.5294   \\
32   &3.1336  & 6.3135   &9.5285  &12.8170  &16.2084  &19.7245   \\
64   &3.1308  & 6.2907   &9.4503  &12.6275  &15.8287  &19.0608   \\\hline
Order  &1.94  & 1.96   &1.98  & 2.00  & 2.00  & 1.99    \\\hline
Exact &3.1299  &  6.2831   & 9.4248  & 12.5664  & 15.7080   &18.8496  \\
\hline
\end{tabular}
\label{tabla1}
\end{center}
\end{table}

The order of convergence is clearly $\mathcal{O}(h^{2})$,
which is expectable according to Theorem~\ref{cotadoblepandeo}
and due the smoothness of the eigenfunctions for this configuration
of the problem. Moreover, the nature of the meshes and 
the fact that we are allowing small edges for the polygons,
does not affect the order of convergence and no spurious eigenvalues were found.

In the next test, we will study the effects of the
stabilization~\eqref{20} in the computation of the spectrum.
We will consider the same physical configuration as in the previous test. 
Since the stabilization depends on the size of the 
element $K$ (see~\eqref{20}), we will compute the
first six eigenvalues for different values $h_K^{\alpha}$
using the family of meshes $\CT_h^2$.


\begin{table}[H]
\begin{center}
\caption{Test2. The lowest computed eigenvalues $\l_{hi}$, $1\le i\le6$
for different $h_{\E}^{\alpha}$, $\alpha=\{1/2, 3/4, 1, 5/4, 3/2\}$ 
with meshes $\CT_{h}^{2}$.}
\vspace{0.3cm}
\begin{tabular}{|c|c|c|c|c|c|c|c|}
\hline
$h_{\E}^{\alpha}$& $N$ & $\l_{h1}$ & $\l_{h2}$ & $\l_{h3}$ & 
 $\l_{h4}$ & $\l_{h5}$& $\l_{h6}$  \\
\hline
&8 &   3.1906  &  6.8000  & 11.2387 &  17.2492  & 24.8335  & 34.3717\\
 &  16 &    3.1456 &   6.4174 &   9.8977 &  13.7232 &  18.0174 &  22.9781\\
$h_{\E}^{1/2} $   & 32 &    3.1338 &   6.3154  &  9.5349 &  12.8325  & 16.2400  & 19.7798\\
  & 64 &    3.1308  &  6.2909  &  9.4511  & 12.6295  & 15.8327  & 19.0678\\
 \hline
 & Order   &2.00   &2.02  & 2.04  & 2.08  & 2.07  & 2.06  \\ \hline
& Exact &3.1299  &  6.2831 &   9.4248 &  12.5664 &  15.7080 &  18.8496  \\\hline
$h_{\E}^{\alpha}$& $N$ & $\l_{h1}$ & $\l_{h2}$ & $\l_{h3}$ & 
 $\l_{h4}$ & $\l_{h5}$& $\l_{h6}$  \\
\hline
&8 &   3.1875 &   6.7730 &  11.1386  & 16.9596  & 24.2790 &  33.4292\\
 &  16 &  3.1451 &   6.4133 &   9.8829 &  13.6861 &  17.9400 &  22.8346\\
$h_{\E}^{3/4} $   &  32  &  3.1337 &   6.3148 &   9.5331 &  12.8282  & 16.2313 &  19.7645\\
 &  64   & 3.1308  &  6.2908  &  9.4509  & 12.6290  & 15.8317  & 19.0662\\ \hline
&  Order   &1.98   &  2.00  &          2.02 &            2.05 &             2.04&             2.03\\\hline
& Exact &3.1299  &  6.2831 &   9.4248 &  12.5664 &  15.7080 &  18.8496  \\
 \hline
 $h_{\E}^{\alpha}$& $N$ & $\l_{h1}$ & $\l_{h2}$ & $\l_{h3}$ & 
 $\l_{h4}$ & $\l_{h5}$& $\l_{h6}$  \\
\hline
&8    &3.1820  & 6.7247  &10.9616  &16.4556  &23.3067  &31.7699  \\
&16   &3.1441  & 6.4043   &9.8511  &13.6063  &17.7749  &22.5294   \\
$h_{E} $&32   &3.1336  & 6.3135   &9.5285  &12.8170  &16.2084  &19.7245   \\
&64   &3.1308  & 6.2907   &9.4503  &12.6275  &15.8287  &19.0608   \\\hline
& Order  &1.94  & 1.96   &1.98  & 2.00  & 2.00  & 1.99    \\\hline
& Exact  &3.1299  &  6.2831   & 9.4248  & 12.5664  & 15.7080   &18.8496  \\
\hline
$h_{\E}^{\alpha}$& $N$ & $\l_{h1}$ & $\l_{h2}$ & $\l_{h3}$ & 
 $\l_{h4}$ & $\l_{h5}$& $\l_{h6}$  \\
\hline
&8    &3.1728   & 6.6440   &10.6696  & 15.6472  & 21.7299  & 29.0670\\
 & 16  &  3.1420 &   6.3861 &   9.7866 &  13.4455 &  17.4448  & 21.9242\\
$h_{\E}^{5/4} $    & 32  &  3.1332 &   6.3101 &   9.5170 &  12.7889 &  16.1512 &  19.6248\\
  & 64  & 3.1307   & 6.2901   & 9.4483  & 12.6228  & 15.8194   &19.0444\\\hline
& Order  &1.88  &         1.90 &            1.91&            1.93 &             1.92 &             1.91    \\\hline
& Exact  &3.1299  &  6.2831   & 9.4248  & 12.5664  & 15.7080   &18.8496  \\
\hline
$h_{\E}^{\alpha}$& $N$ & $\l_{h1}$ & $\l_{h2}$ & $\l_{h3}$ & 
 $\l_{h4}$ & $\l_{h5}$& $\l_{h6}$  \\
\hline
&8&   3.1585   & 6.5223 &  10.2377  & 14.5052 &  19.4815 &  25.2098\\
  & 16&   3.1382 &   6.3532 &   9.6706 &  13.1590 &  16.8635 &  20.8743\\
$h_{\E}^{3/2} $   &32 &   3.1323  &  6.3026 &   9.4912  & 12.7263  & 16.0248 &  19.4044\\
  & 64 &  3.1305  &  6.2885  &  9.4430  & 12.6098  & 15.7937  & 18.9992\\\hline
&Order  &1.81   &  1.83   &          1.83 &             1.83 &             1.82 &              1.81\\\hline
&Exact  &3.1299  &  6.2831   & 9.4248  & 12.5664  & 15.7080   &18.8496  \\
\hline
\end{tabular}
\label{tablan}
\end{center}
\end{table}

We observe from the results of Table~\ref{tablan} that the method
converges to the exact eigenvalues with an optimal
quadratic order and no spurious eigenvalues were found for any chosen
stability parameter $h_K^{\alpha}$. We remark that these results
are also valid for other type of family polygonal meshes allowing small edges.



\subsection{Rotated T domain}

In the following test we will consider a non-convex
domain which we call \emph{rotated T} an it is defined
by $\Omega_T:=(-0.5, 0.5)\times(-0.5, 0)\cup (-0.25,0.25)\times(0,1)$
with boundary condition $\Gamma_0=\partial \O_T$.
This non-convex domain presents two reentrant angles of the same size
$\omega=\frac{3\pi}{2}$ (cf. Figure~\ref{FIG:meshes2}),
and as a consequence, the eigenfunctions
of this problem may present singularities.
%
More precisely,
the Sobolev exponent for the eigenfunctions
is $2/3$ (cf. Remark~\ref{angleome}), so that the eigenfunctions
will belong to $\HurO$ for all $r<2/3$, but in general not to $H^{1+\frac{2}{3}}(\O)$.
Therefore, according to Theorem~\ref{cotadoblepandeo},
the convergence rate for the eigenvalues should be $\vert\lambda-\lambda_h\vert\approx h^{4/3}$.

In Figure~\ref{FIG:meshes2}, we present the meshes
that we will consider for this numerical test.
We note that the families of polygonal meshes $\CT_h^3$, $\CT_h^4$ and $\CT_h^5$
have been obtained by gluing two different polygonal meshes at $x=0$.
It can be seen that very small edges compared with the element
diameter appears on the interface of the resulting meshes.

\begin{figure}[H]
\begin{center}
\begin{minipage}{15cm}
\centering\includegraphics[height=4.5cm, width=4.5cm]{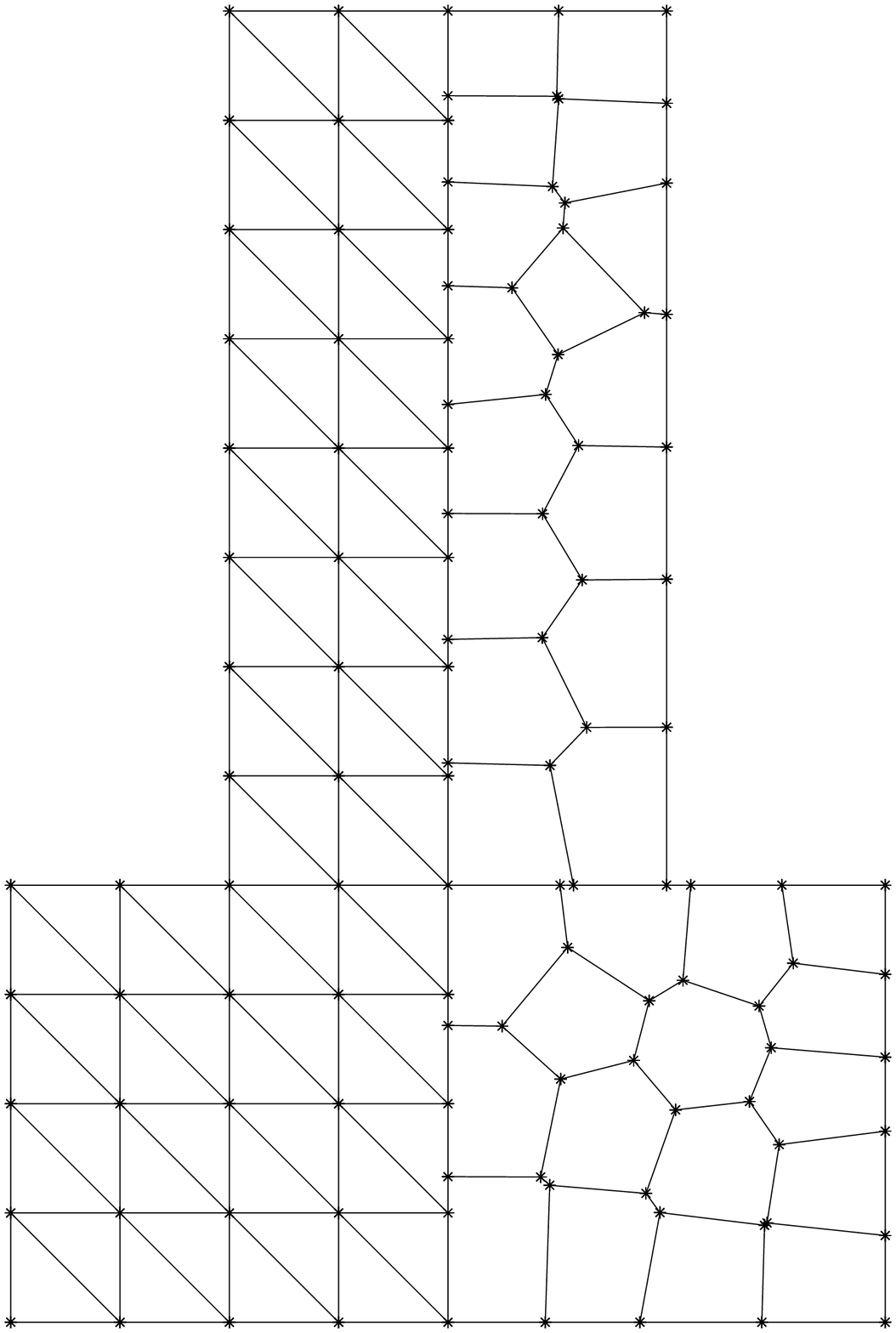}
\centering\includegraphics[height=4.5cm, width=4.5cm]{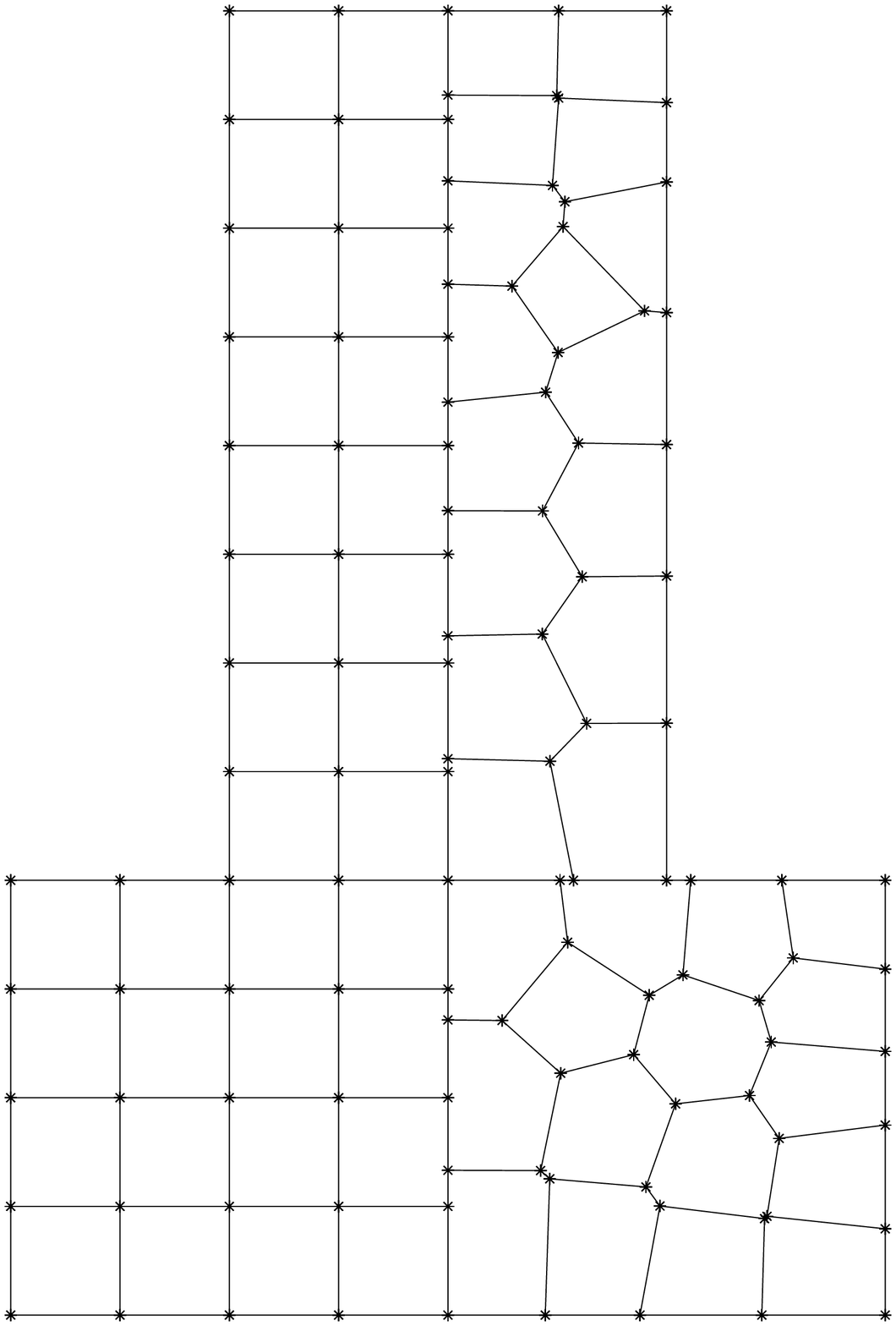}
\centering\includegraphics[height=4.5cm, width=4.5cm]{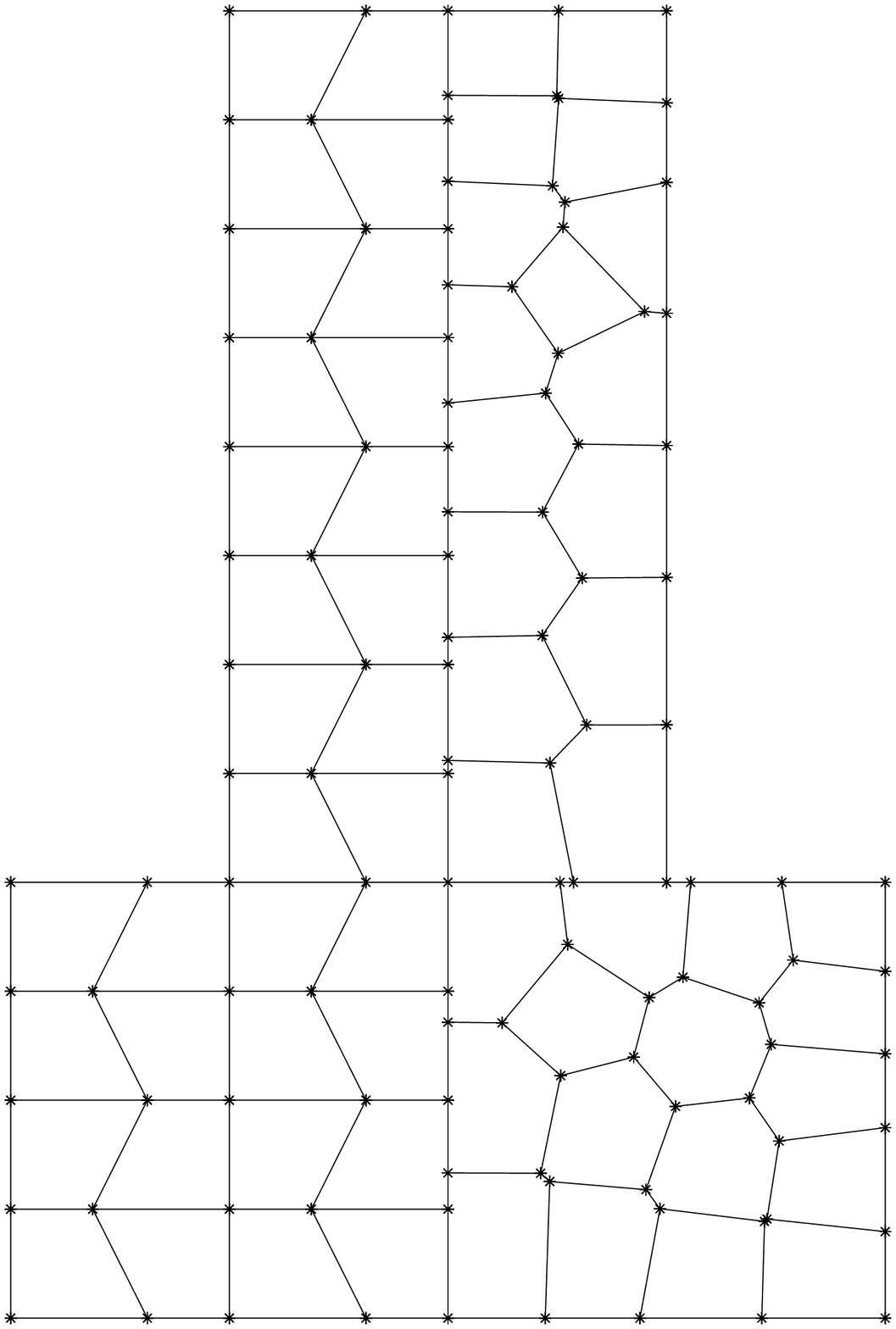}
\end{minipage}
\caption{Sample meshes with small edges. From left to right: $\CT_h^3$, $\CT_h^4$ and $\CT_h^5$, for $N=8$.}
\label{FIG:meshes2}
\end{center}
\end{figure}

In Table~\ref{tabla3}, we report the computed eigenvalues
and the corresponding convergence rates,
the last row, that we called 'Extrap.', corresponds
to extrapolated values obtained with a least square fitting of the form
\begin{equation}
\label{eq:fitting}
\lambda_{hi}\approx\lambda_{i}+C_ih^{\alpha_i},
\end{equation}
where $\alpha_i$ is the approximated rate of convergence
of each $\lambda_i$, with $i\in\mathbb{N}$.

\begin{table}[H]
\begin{center}
\caption{Test 2. Computed lowest eigenvalues $\l_{h}^{i}$, $1\le
i\le4$, on different polygonal meshes.}
\vspace{0.3cm}
\begin{tabular}{|c|c|cccc|c|c|c|}
\hline
$\CT_h$   & $\l_{h}^{i}$ & $N=16$ & $N=30$ & $N=62$ & $N=130$ &
Order & Extrap. \\
\hline
          & $\l_{h}^{1}$ &0.5196  & 0.5157  & 0.5140 & 0.5134 &   1.41 &  0.5130  \\
$\CT_h^3$ & $\l_{h}^{2}$ & 1.2743 &  1.2622 & 1.2570 &  1.2552  &  1.48  & 1.2543\\
         & $\l_{h}^{3}$& 2.5567 & 2.5263 & 2.5146  & 2.5111  &   1.65 &  2.5096\\
           & $\l_{h}^{4}$& 3.1923 & 3.1556  & 3.1458 &  3.1437 &  2.19  & 3.1432\\
\hline
          & $\l_{h}^{1}$ &  0.5209   & 0.5163  &  0.5142 &  0.5135   & 1.41 &  0.5131 \\
$\CT_h^4$& $\l_{h}^{2}$ &1.2793 &  1.2641&  1.2577 & 1.2555 &   1.51 & 1.2545 \\
          & $\l_{h}^{3}$&2.5659 &  2.5296  & 2.5158  &  2.5115   & 1.66  & 2.5098 \\
          & $\l_{h}^{4}$&   3.2144&  3.1616  & 3.1474& 3.1441  & 2.16 & 3.1434\\
 \hline
          & $\l_{h}^{1}$ &0.5209   & 0.5163  &   0.5142 &  0.5135 &   1.41 &  0.5131  \\
$\CT_h^5$& $\l_{h}^{2}$  & 1.2795 &   1.2641 & 1.2577   & 1.2555  &   1.52 & 1.2545  \\
          & $\l_{h}^{3}$ &  2.5663 & 2.5296 &2.5158&   2.5115 &   1.66 &   2.5098\\
          & $\l_{h}^{4}$ &   3.2143 &   3.1616 & 3.1474 & 3.1441&   2.16 &  3.1434\\
\hline
 \end{tabular}
\label{tabla3}
\end{center}
\end{table}

We observe from Table~\ref{tabla3} that 
for the first Steklov eigenvalue the method
converges with order close to $4/3$ which corresponds
to the Sobolev regularity for the Steklov problem on $\O_T$ (non-convex domain).
We also note that the method converges 
larger orders for the rest of the Steklov eigenvalues.

In Figure~\ref{FIG:Tmodes} we present plots
for the first four eigenfunctions for the Steklov problem in the \emph{rotated T} domain,
computed with $\CT_h^5$ and $N=30$.
\begin{figure}[H]
\begin{center}
\begin{minipage}{5cm}
\centering\includegraphics[height=5cm, width=5cm]{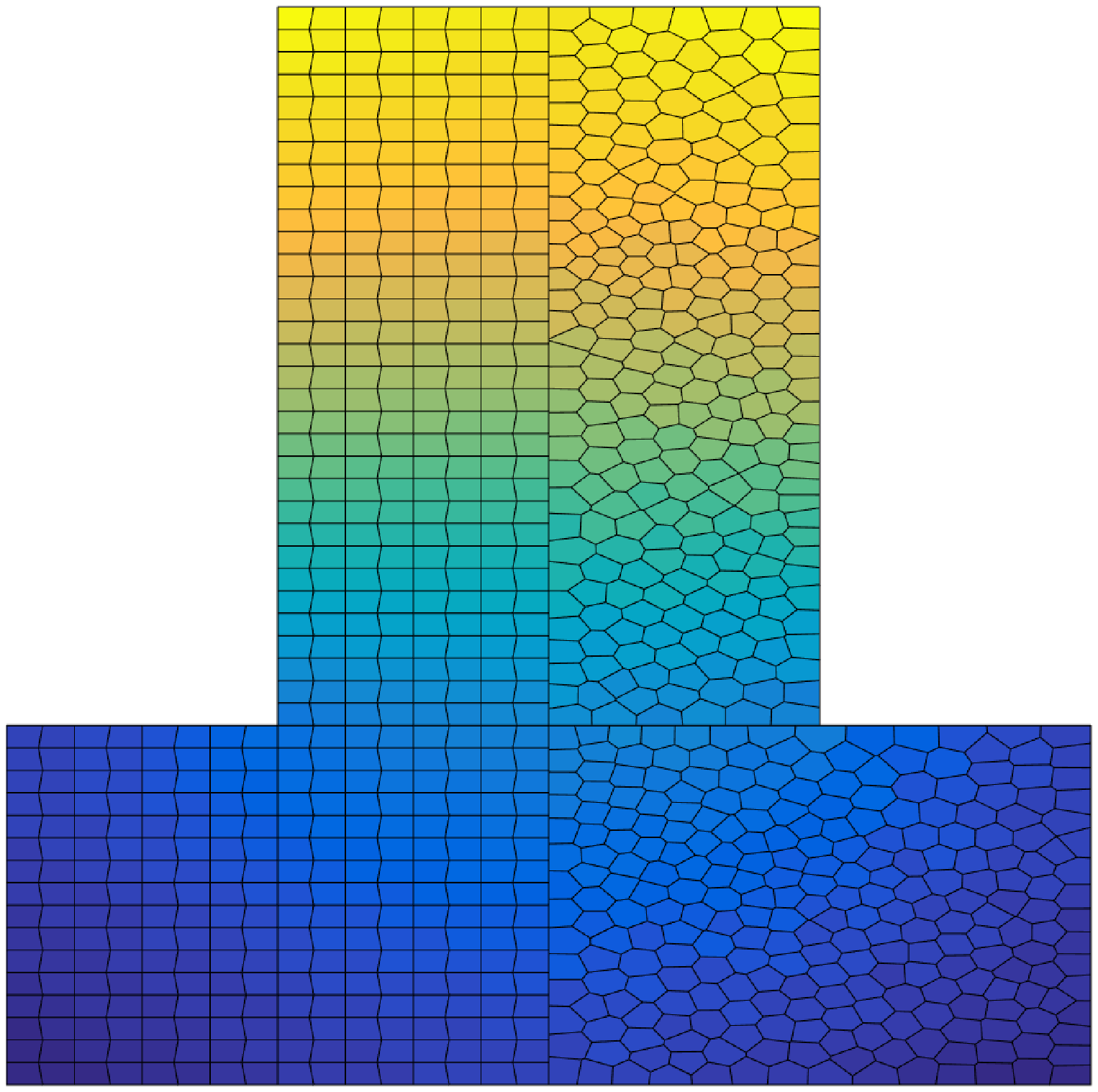}
\end{minipage}
\begin{minipage}{5cm}
\centering\includegraphics[height=5cm, width=5cm]{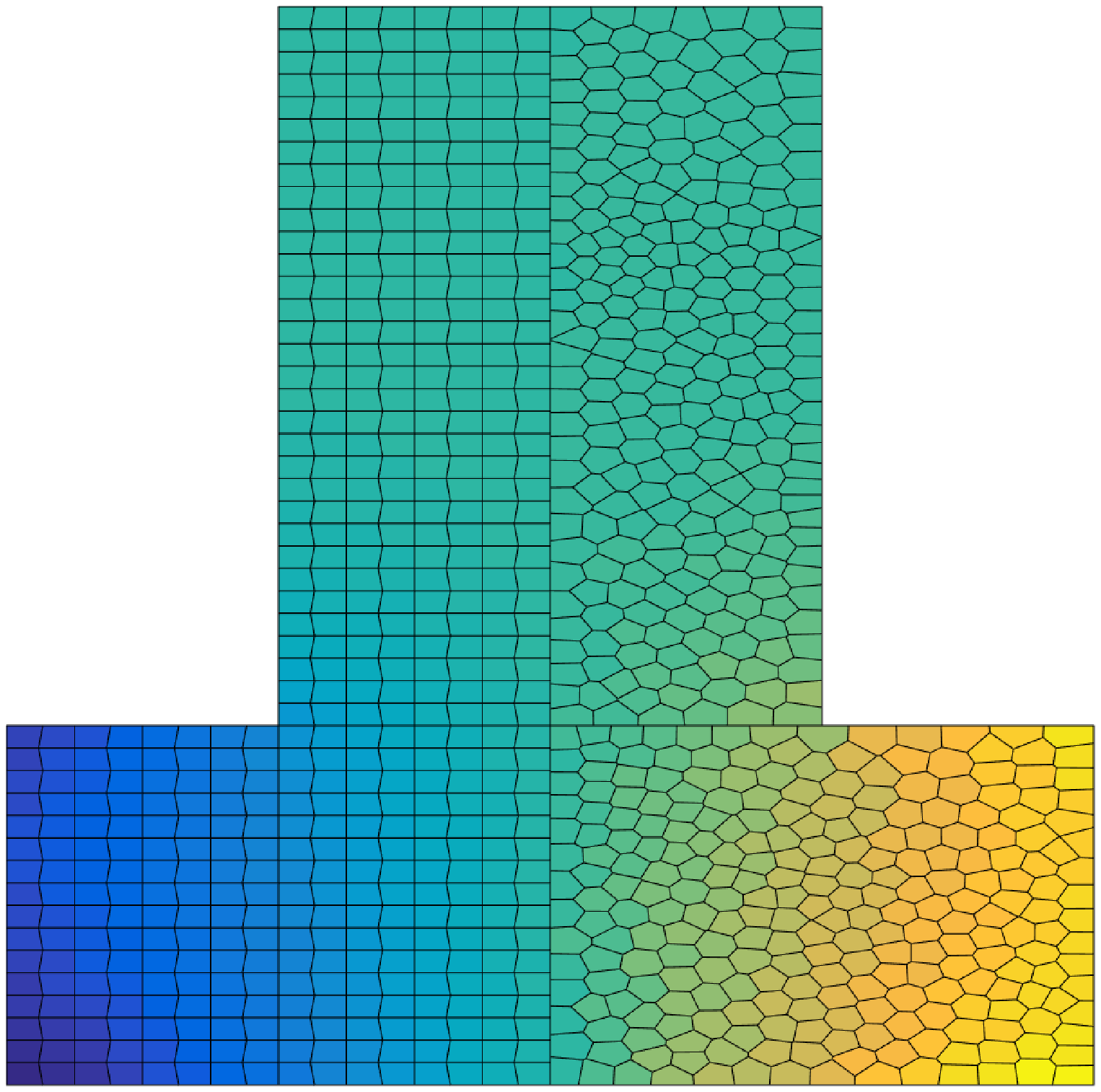}
\end{minipage}
\begin{minipage}{5cm}
\centering\includegraphics[height=5cm, width=5cm]{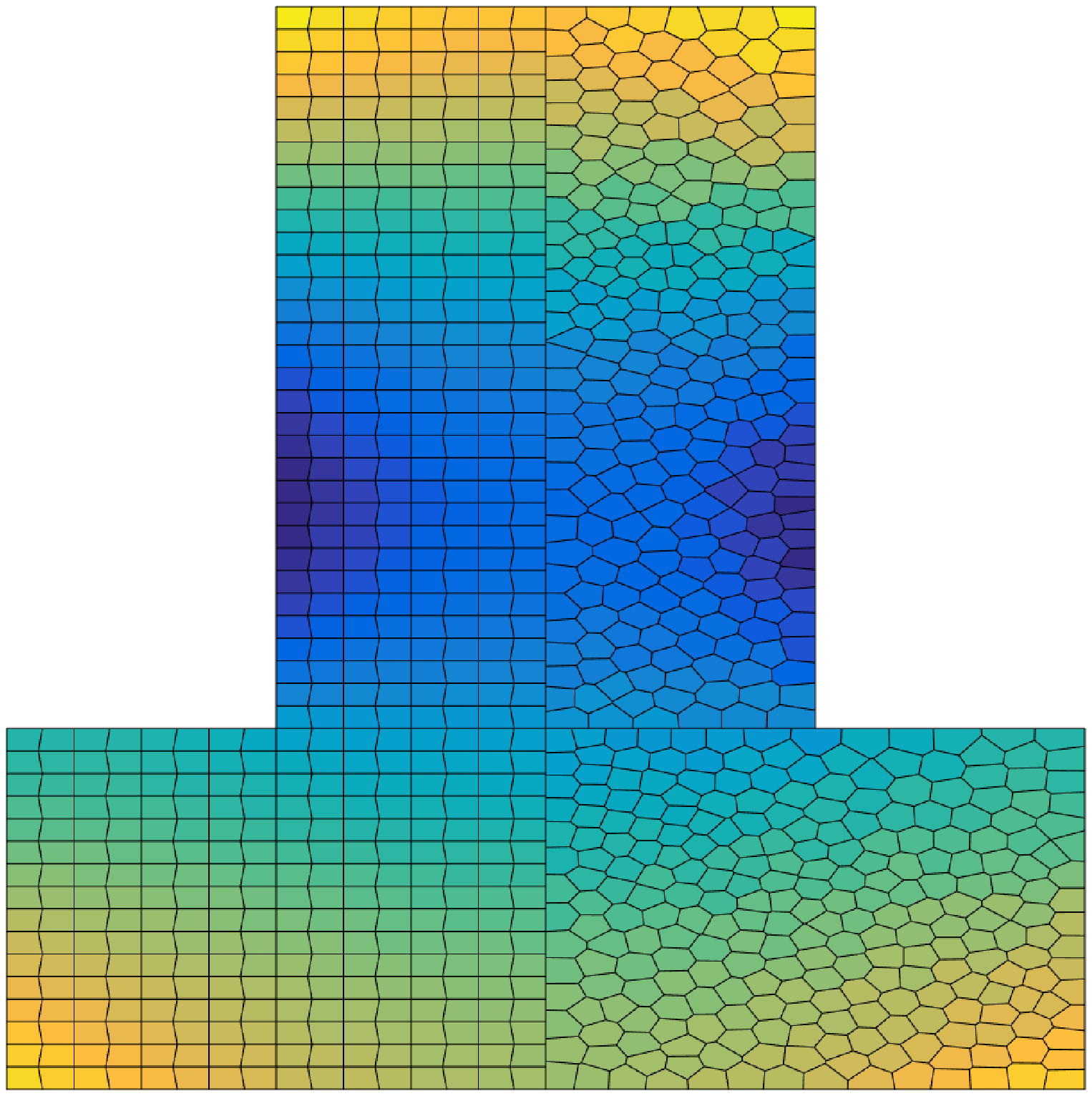}
\end{minipage}
\begin{minipage}{5cm}
\centering\includegraphics[height=5cm, width=5cm]{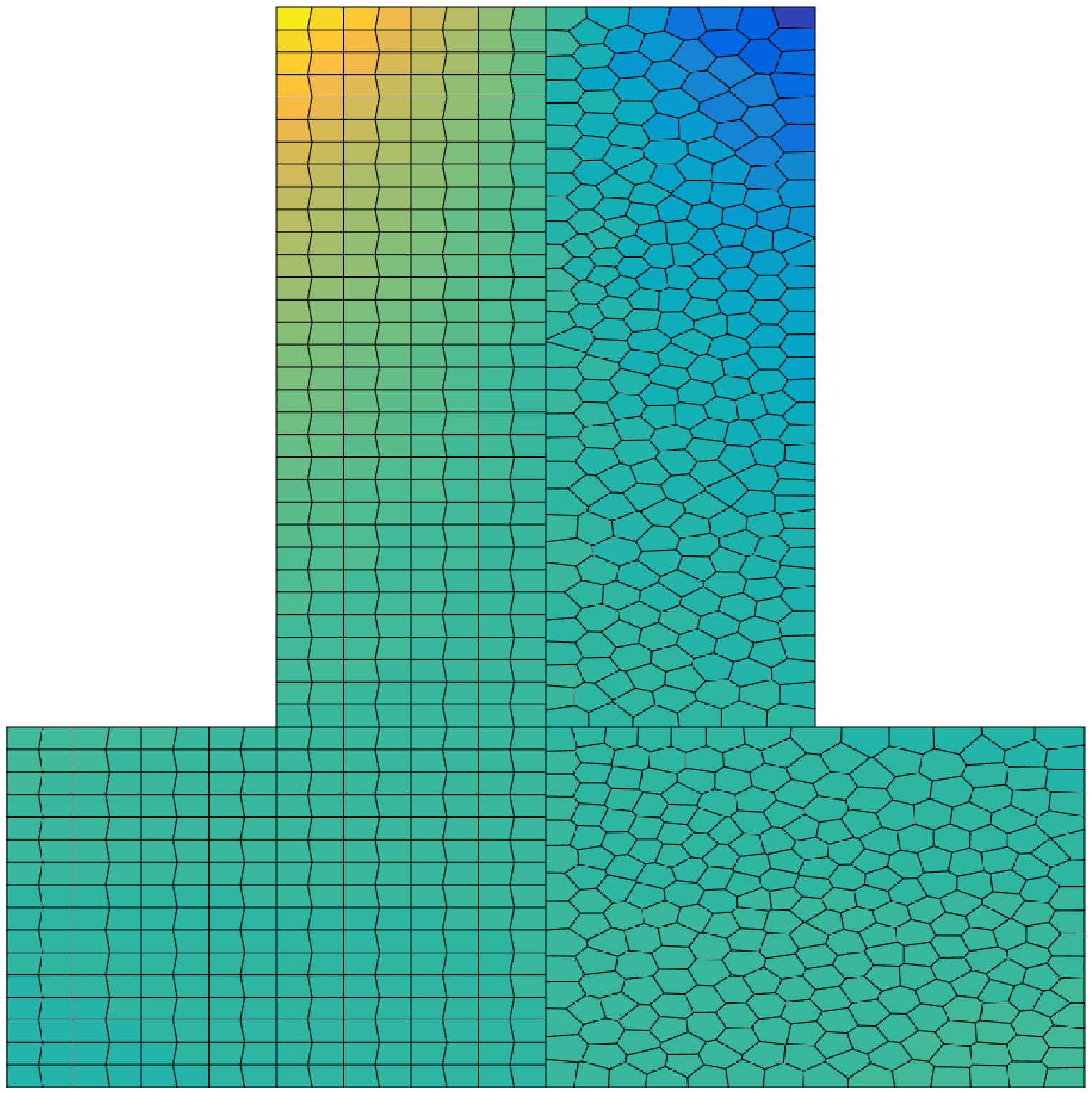}
\end{minipage}
\caption{From top left to bottom right, plots of the first four
eigenfunctions for the \emph{rotated T} domain, computed with $\CT_h^5$. }
		\label{FIG:Tmodes}
	\end{center}
\end{figure}

\subsection{L-shaped domain}
In this numerical example we test the properties of the
proposed method on an L-shaped domain:
$\O_L:=(0,1)\times(0,1)\backslash [0.5,1)\times[0.5,1)$
with $\Gamma_{0}=\partial\O_L$. For this test, we will
adopt a refinement with hanging nodes, which implies to
consider once again polygons with small edges.
More precisely, this test is focused to validate
the use of refined meshes as a tool to handle solutions
with corner singularities. With this purpose, we have
considered two families of meshes, namely: $\CT_h^6$ (see
upper left picture in Figure~\ref{FIG:LCuadRef})
and $\CT_h^{6,\ell}$. The initial uniform mesh $\CT_h^{6}$
has $N=32$ elements on each edge and the last one has $N = 128$
elements on each edge.


On the other hand, the mesh $\CT_h^{6,\ell}$ is obtained
by refining a patch around the re-entrant corner of $\Omega_L$,
starting from an initial uniform quadrilateral mesh $\CT_h^{6,0}$,
which corresponds to the first mesh of $\CT_h^6$. 
The procedure consists in to split each element which belongs to the region:

$$R_{\ell}:=\left\{(x,y)\in \mathbb{R}^2:  |x-1/2|\leq \frac{6}{N}2^{1-\ell}
\quad \mbox{and} \quad  |y-1/2|\leq  \frac{6}{N}2^{1-\ell}\right\}
\cap \overline{\Omega}_L \qquad \ell=1,2,\ldots,\widehat{\ell},$$  
into three quadrilaterals by connecting the barycenter of the element
with the midpoint of each edge, where $\widehat{\ell}$ is the number
of meshes to refine, with the convention that $\CT_h^{6,0}:=\CT_h^{6}$
(the initial mesh with $N=32$). Note that although this process
is initiated with a quadrilateral mesh, the successively created meshes
will contain other kind of convex polygons as can be appreciated
in Figure~\ref{FIG:LCuadRef}.

\begin{figure}[H]
	\begin{center}
		\begin{minipage}{5cm}
			\centering\includegraphics[height=5cm, width=5cm]{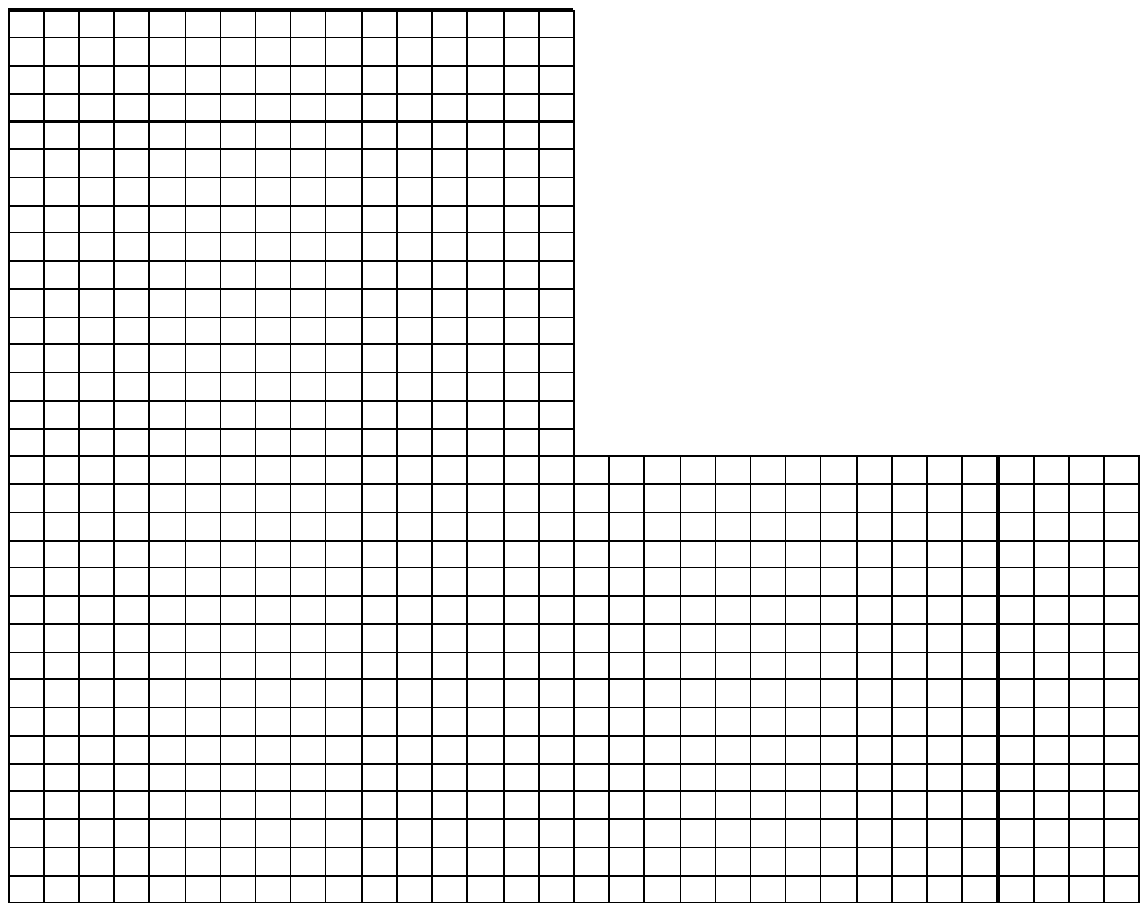}
		\end{minipage}
		\begin{minipage}{5cm}
			\centering\includegraphics[height=5cm, width=5cm]{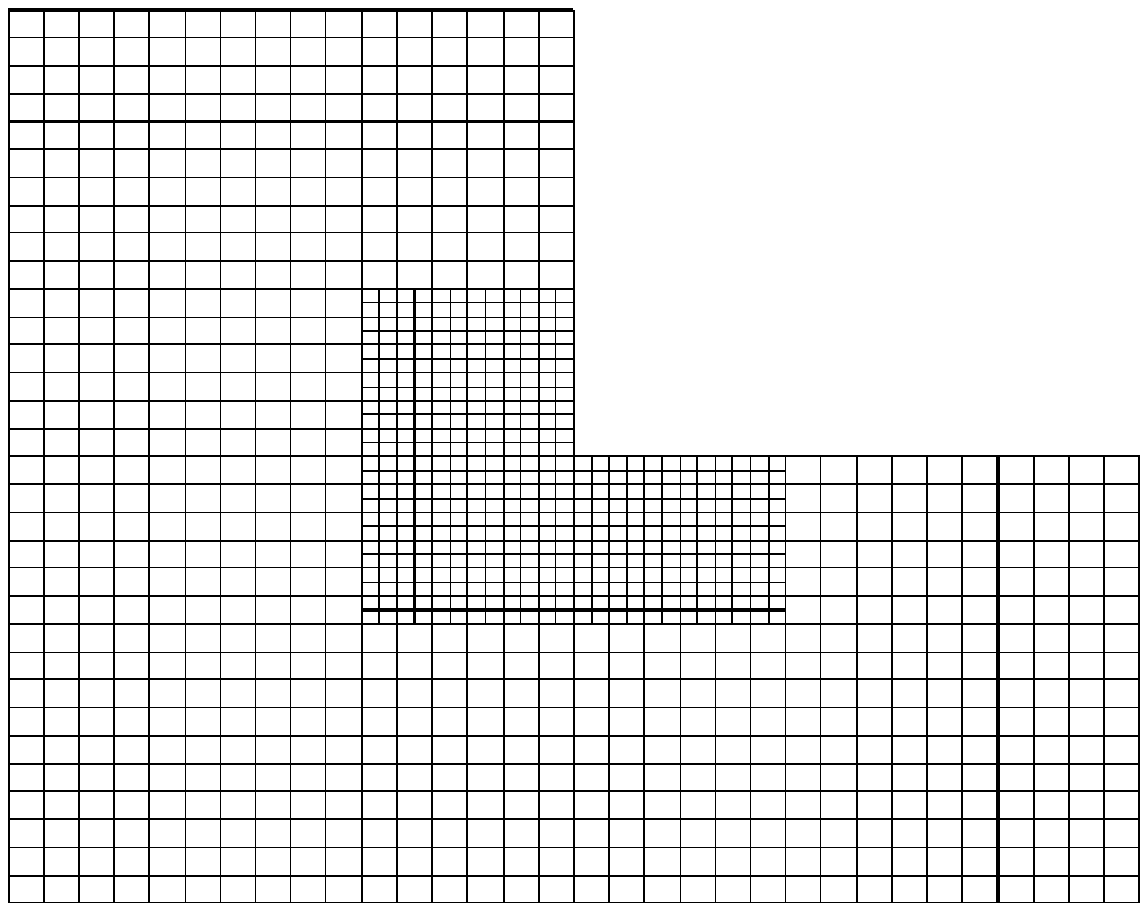}
		\end{minipage}
		\begin{minipage}{5cm}
			\centering\includegraphics[height=5cm, width=5cm]{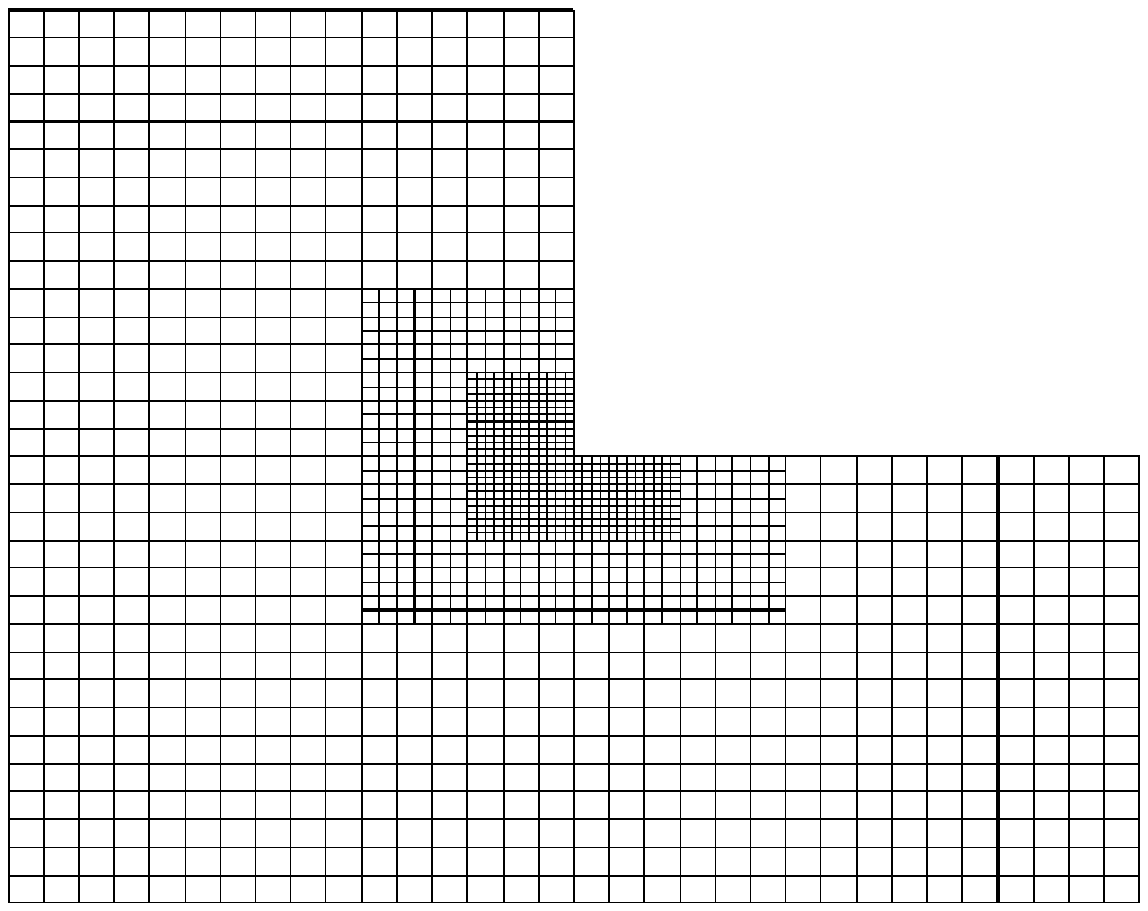}
		\end{minipage}
		\begin{minipage}{5cm}
			\centering\includegraphics[height=5cm, width=5cm]{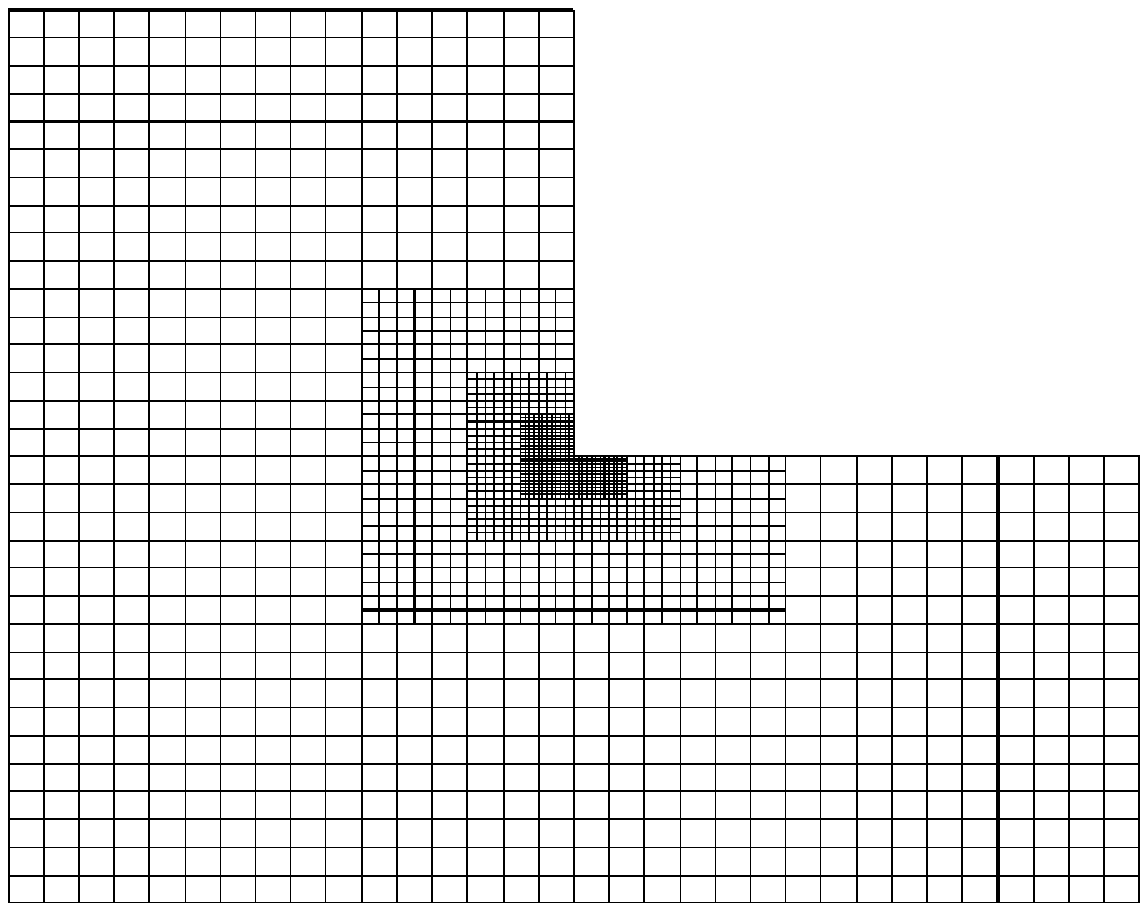}
		\end{minipage}
\caption{Sample meshes: inicial mesh $\CT_{h}^{6}$ with $N=16$ (top left),
	$\CT_{h}^{6,1}$ (top right),  $\CT_{h}^{6,2}$ (bottom left) and $\CT_{h}^{6,3}$ (bottom right).}
		\label{FIG:LCuadRef}
	\end{center}
\end{figure}


\begin{table}[H]
\begin{center}
\caption{Test 3. Test with an L-shaped domain. Number of degrees of freedom,
computed lowest eigenvalue and error, by using uniform square
meshes and polygonal meshes with hanging nodes.}
\vspace{0.3cm}
\begin{tabular}{|c|c|c|c|c|c|}
\hline
$\CT_h$   & Dofs&$\l_{h}^{1}$ &Error  \\
\hline
  &  833  &  0.78073215782 &  0.00628166703\\
 $\CT_{h}^{6}$ & 3201  & 0.77689137854&   0.00244088775\\
  & 12545 &0.77539520174 &  0.00094471094\\\hline
    ref. & &   0.77445049080 & -- \\
   \hline
$ \CT_{h}^{6,0}$   &833  & 0.78073215782&  0.00628166703\\
 $\CT_{h}^{6,1} $  &1181 &0.77728198716&  0.00283149637\\
$ \CT_{h}^{6,2} $  &1529  &0.77598279448& 0.00153230369\\
 $\CT_{h}^{6,3} $  &1877  &0.77548305066 & 0.00103255987\\
 $\CT_{h}^{6,4}$  & 2232& 0.77528749982&   0.00083700903\\\hline
    ref. & & 0.77445049080 & -- \\
\hline
 \end{tabular}
\label{tablaL}
\end{center}
\end{table}

Table~\ref{tablaL} reports the lowest Steklov eigenvalue
computed on an L-shaped domain with the method analyzed in this paper
with different polygonal meshes. The table also includes
the corresponding 'Errors' which have been obtained against a reference
value 'ref.' which corresponds to extrapolated values obtained
with a least square fitting on finer uniform meshes.



It can be seen from Table~\ref{tablaL} that
the reported errors are similar
in the last row of each mesh; however,
the dofs in the case of corner-refined meshes are much less
than the case of uniform meshes.
Therefore, we conclude that the possibility of using 
small edges in the polygons of the mesh, 
allow us easier refinements near edges and/or corners
of the domain to handle solutions with
corner singularities.


\section*{Acknowledgments}
FL was partially supported by
CONICYT-Chile through FONDECYT Postdoctorado project 3190204 (Chile).
DM was partially supported by CONICYT-Chile through FONDECYT project
1180913 (Chile) and by CONICYT-Chile through the project AFB170001
of the PIA Program: Concurso Apoyo a Centros Cient\'ificos y
Tecnol\'ogicos de Excelencia con Financiamiento Basal. GR  was supported by
CONICYT-Chile through FONDECYT project 11170534 (Chile).
IV was partially supported by BASAL project CMM,
Universidad de Chile (Chile).


\bibliographystyle{amsplain}

\begin{thebibliography}{99}





\bibitem{AN2020} 
\textsc{D. Adak and S. Natarajan},
\textit{Virtual element method for a nonlocal
elliptic problem of Kirchhoff type on polygonal meshes},
Comput. Math. Appl., 79, (2020), pp. 2856--2871.





\bibitem{ABSV2016} 
\textsc{P. F. Antonietti, L. Beir\~ao da Veiga, S. Scacchi and M. Verani},
\textit{A $C^1$ virtual element method for the Cahn--Hilliard
equation with polygonal meshes},
SIAM J. Numer. Anal., 54, (2016), pp. 36--56.



\bibitem{A_M2AN2004} 
\textsc{M. G. Armentano}, 
\textit{The effect of reduced integration in the
Steklov eigenvalue problem},
ESAIM Math. Model. Numer. Anal., 38, (2004), pp. 27--36.

\bibitem{AP_APNUM2009} 
\textsc{M. G. Armentano and C. Padra}, 
\textit{A posteriori error estimates for the Steklov eigenvalue
problem},
Appl. Numer. Math., 58, (2008), pp. 593--601.

\bibitem{BO} 
\textsc{I. Babu\v{s}ka and J. Osborn}, 
\textit{Eigenvalue problems}, 
in \textit{Handbook of Numerical Analysis}, Vol. II, 
P.G. Ciarlet and J.L. Lions, eds., 
North-Holland, Amsterdam, 1991, pp. 641--787.


\bibitem{BRS03} 
\textsc{A. Berm\'udez, R. Rodr\'iguez and D. Santamarina},
\textit{Finite element computation of sloshing modes in containers with 
elastic baffle plates},
Internat. J. Numer. Methods Engrg., 56, (2003), pp. 447--467. 

\bibitem{BBCMMR2013} 
\textsc{L. Beir\~ao da Veiga, F. Brezzi, A. Cangiani, G. Manzini, L. D.
Marini and A. Russo},
\textit{Basic principles of virtual element methods},
Math. Models Methods Appl. Sci., 23, (2013), pp. 199--214.



\bibitem{BDR2017} 
\textsc{L. Beir\~ao da Veiga, F. Dassi and A. Russo},
\textit{High-order virtual element method on polyhedral meshes},
Comput. Math. Appl., 74, (2017), pp. 1110--1122.



\bibitem{BLMbook2014} 
\textsc{L. Beir\~ao da Veiga, K. Lipnikov and G. Manzini},
\textit{The Mimetic Finite Difference Method for Elliptic Problems},
Springer, MS\&A, vol. {\bf 11}, 2014.

\bibitem{BLM2015} 
\textsc{L. Beir\~ao da Veiga, C. Lovadina and D. Mora},
\textit{A virtual element method for elastic and
inelastic problems on polytope meshes},
Comput. Methods Appl. Mech. Engrg., 295, (2015) pp. 327--346.

\bibitem{BLR17} 
\textsc{L. Beir\~ao da Veiga, C. Lovadina and A. Russo},
\textit{Stability analysis for the virtual element method},
Math. Models Methods Appl. Sci., 27, (2017), pp. 2557--2594.

\bibitem{BLV-M2AN} 
\textsc{L. Beir\~ao da Veiga, C. Lovadina  and G. Vacca},
\textit{Divergence free virtual elements for the Stokes
problem on polygonal meshes},
ESAIM Math. Model. Numer. Anal., 51, (2017), pp. 509--535.




  
\bibitem{BM20} 
\textsc{L. Beir\~ao da Veiga and G. Vacca},
\textit{Sharper error estimates for virtual
elements and a bubble-enriched version},
arXiv:2005.12009 [math.NA].
  
  

\bibitem{BBBPS2016} 
\textsc{M. F. Benedetto, S. Berrone, A. Borio, S. Pieraccini and S. Scial\`o},
\textit{Order preserving SUPG stabilization for the virtual element
formulation of advection--diffusion problems},
Comput. Methods Appl. Mech. Engrg., 311, (2016), pp. 18--40.
  

\bibitem{BrOs1972} 
\textsc{J. H. Bramble and J. E. Osborn}, 
\textit{Approximation of Steklov eigenvalues of non-selfadjoint second
order elliptic operators}, 
in \textit{The Mathematical Foundations of the Finite Element Method
with Applications to Partial Differential Equations}, 
A. K. Aziz, ed., Academic Press, New York, 1972, pp. 387--408.

\bibitem{BS-2008} 
\textsc{S. C. Brenner and R. L. Scott},
\textit{The Mathematical Theory of Finite Element Methods},
Springer, New York, 2008.

\bibitem{BrS}
\textsc{S. C. Brenner and L. Y. Sung},
\textit{Virtual element methods on meshes with small edges or faces},
Math. Models Methods Appl. Sci., 28, (2018), pp. 1291--1336.




\bibitem{CG2017} 
\textsc{E. C\'aceres and G. N. Gatica},
\textit{A mixed virtual element method for the
pseudostress-velocity formulation of the Stokes problem},
IMA J. Numer. Anal., 37, (2017), pp. 296--331.

\bibitem{CM79} 
\textsc{J. Canavati and A. Minsoni},
\textit{A discontinuous Steklov problem with an application to water waves},
J. Math. Anal. Appl., 69 (1979), pp. 540--558.

\bibitem{CGH14} 
\textsc{A. Cangiani, E. H. Georgoulis and P. Houston},
\textit{$hp$-version discontinuous Galerkin methods
on polygonal and polyhedral meshes},
Math. Models Methods Appl. Sci., 24, (2014), pp. 2009--2041.


\bibitem{CGPS2017}
\textsc{A. Cangiani, E. H. Georgoulis, T. Pryer and O.J. Sutton},
\textit{A posteriori error estimates for the virtual element method},
Numer. Math., 137, (2017), pp. 857-893.





\bibitem{CGMMV} 
\textsc{O. \v{C}ert\'ik,  F. Gardini, G. Manzini, L. Mascotto  and G. Vacca},
\textit{The p- and hp-versions of the virtual element
method for elliptic eigenvalue problems},
Comput. Math. Appl., 79, (2020), pp. 2035--2056.






\bibitem{CY96} 
\textsc{Y. S. Choun and C. B. Yun},
\textit{Sloshing characteristics in rectangular tanks with a submerged block},
Comput. Struct., 61, (1996), pp. 401--413.


\bibitem{DrAA2011} 
\textsc{A. Dello Russo and A. Alonso},
\textit{A posteriori error estimates for nonconforming approximations of
Steklov eigenvalue problems},
Comput. Math. Appl., 62, (2011), pp. 4100--4117.




\bibitem{DDbook2020} 
\textsc{D. Di Pietro and J. Droniou},
\textit{The Hybrid High-Order Method for Polytopal Meshes - Design, Analysis and Applications},
Springer, MS\&A, vol. {\bf 19}, 2020.



\bibitem{GM-ima2011} 
\textsc{E. M. Garau and P. Morin}, 
\textit{Convergence and quasi-optimality of adaptive FEM for Steklov
eigenvalue problems},
IMA J. Numer. Anal., 31, (2011), pp. 914--946.

\bibitem{GMV} 
\textsc{F. Gardini, G. Manzini,  and G. Vacca},
\textit{The nonconforming virtual element method for eigenvalue problems},
ESAIM Math. Model. Numer. Anal., 53, (2019), pp. 749--774.





\bibitem{GV} 
\textsc{F. Gardini  and G. Vacca},
\textit{Virtual element method for second-order elliptic eigenvalue problems},
IMA J. Numer. Anal. 38, (2018), pp. 2026-2054.




\bibitem{GR} 
\textsc{V. Girault and P. A. Raviart},
\textit{Finite Element Methods for Navier-Stokes Equations},
Springer-Verlag, Berlin, 1986.

\bibitem{G} 
\textsc{P. Grisvard}, 
\textit{Elliptic Problems in Non-Smooth Domains}, 
Pitman, Boston, 1985.

\bibitem{K} 
\textsc{T. Kato}, 
\textit{Perturbation Theory for Linear Operators},
Springer Verlag, Berlin, 1995.

\bibitem{LLX_apmath2013} 
\textsc{Q. Li, Q. Lin and H. Xie},
\textit{Nonconforming finite element approximations of the Steklov
eigenvalue problem and its lower bound approximations}, 
Appl. Math., 58, (2013), pp. 129--151.

\bibitem{LSTJSC19} 
\textsc{J. Liu, J. Sun and T. Turner},
\textit{Spectral indicator method for a non-selfadjoint
Steklov eigenvalue problem},
J. Sci. Comput., 79, (2019), pp. 1814--1831.


\bibitem{MPP2018} 
\textsc{L. Mascotto, I. Perugia and A. Pichler},
\textit{Non-conforming harmonic virtual element method:
$h$- and  $p$- versions},
J. Sci. Comput., 77, (2018), pp. 1874--1908.



\bibitem{MM} 
\textsc{J. Meng and L. Mei},
\textit{A linear virtual element method
for the Kirchhoff plate buckling problem}, 
Appl. Math. Lett., 103, (2020), 106188, 8 pp.






\bibitem{MR2019}
{\sc D. Mora and G. Rivera},
\textit{A priori and a posteriori error estimates for a
virtual element spectral analysis for the elasticity equations},
IMA J. Numer. Anal., 40 (2020), pp. 322--357.

\bibitem{MRR1}
{\sc D. Mora, G. Rivera and R. Rodr\'iguez}, {\it A virtual element method
for the Steklov eigenvalue problem},
Math. Models Methods Appl. Sci., 25, (2015), pp. 1421--1445.

\bibitem{MRR2}
{\sc D. Mora, G. Rivera and R. Rodr\'iguez},
{\it A posteriori error estimates for a virtual elements
method for the Steklov eigenvalue problem},
Comp. Math. Appl.,  74, (2017), pp. 2172--2190.
 
 
 
 \bibitem{MV2}
{\sc D. Mora and I. Vel\'asquez}, {\it Virtual element for the
buckling problem of Kirchhoff-Love plates},
Comput. Methods Appl. Mech. Engrg., 360, (2020), 112687, 22 pp.

  \bibitem{PPR15} 
\textsc{I. Perugia, P. Pietra and A. Russo},
\textit{A plane wave virtual element method for the Helmholtz problem},
ESAIM Math. Model. Numer. Anal., 50, (2016), pp. 783--808.
  
  \bibitem{RW} 
\textsc{S. Rjasanow and S. Wei\ss er},
\textit{Higher order BEM-based FEM on polygonal meshes},
SIAM J. Numer. Anal., 50, (2012), pp. 2357--2378.
  
\bibitem{ST04} 
\textsc{N. Sukumar and A. Tabarraei},
\textit{Conforming polygonal finite elements},
Internat. J. Numer. Methods Engrg., 61, (2004), pp. 2045--2066.
  
  
\bibitem{YLL_apnum2009} 
\textsc{Y. Yang, Q. Li and S. Li},
\textit{Nonconforming finite element approximations of the Steklov
eigenvalue problem}, 
Appl. Numer. Math., 59, (2009), pp. 2388--2401.

\bibitem{XIMA13} 
\textsc{H. Xie},
\textit{A type of multilevel method for the Steklov eigenvalue problem},
IMA J. Numer. Anal., 34 (2014), pp. 592--608.




\bibitem{WRR} 
\textsc{P. Wriggers, W. T. Rust and B. D. Reddy},
\textit{A virtual element method for contact},
Comput. Mech., 58, (2016), pp. 1039--1050.


\end{thebibliography}

\end{document}